\documentclass[10pt,a4paper]{amsart}
\usepackage{amssymb,mathtools}
\usepackage[margin=1in]{geometry}
\usepackage{extarrows}
\usepackage{iftex}
\ifLuaTeX
  \usepackage{luatex85}
\fi
\usepackage[all]{xy}
\usepackage{booktabs,tabularx,array,enumitem,microtype}
\usepackage[hidelinks]{hyperref}
\setlist{nosep,leftmargin=2.2em}
\allowdisplaybreaks[2]
\theoremstyle{plain}
\newtheorem{theorem}{Theorem}[section]
\newtheorem{lemma}[theorem]{Lemma}
\newtheorem{proposition}[theorem]{Proposition}
\newtheorem{corollary}[theorem]{Corollary}
\theoremstyle{definition}
\newtheorem{definition}[theorem]{Definition}
\newtheorem{example}[theorem]{Example}
\newtheorem{conjecture}[theorem]{Conjecture}
\newtheorem{problem}[theorem]{Problem}
\theoremstyle{remark}
\newtheorem{remark}[theorem]{Remark}

\begin{document}

\title[Gabriel--Zisman Localizations and (Co)products]{Gabriel--Zisman Localizations, Products, Coproducts, and Product Categories}

\author{Chencheng Zhang}
\address{School of Mathematical Sciences, Shanghai Jiao Tong University, Shanghai 200240, P. R. China}
\email{zhangchencheng@sjtu.edu.cn}
\thanks{The author was supported by the National Natural Science Foundation of China (No.~12131015).}
\begin{abstract}
  This paper studies (co)products and additive structures under Gabriel--Zisman localization $Q:\mathcal A\longrightarrow\mathcal A[S^{-1}]$.
  Examples show that localization may destroy existing (co)products or fail to preserve (co)products that remain, even when $Q$ is additive.
  For a set-indexed family $(\mathcal A_i,S_i)_{i\in I}$, necessary and sufficient conditions are given for the canonical functor $\big(\prod_{i\in I}\mathcal A_i\big)\big[\big(\prod_{i\in I}S_i\big)^{-1}\big]\longrightarrow\prod_{i\in I}\mathcal A_i[S_i^{-1}]$ to be full or faithful.
  These criteria yield compatibility with $I$-indexed (co)products for localizations admitting a calculus of fractions or arising from model categories, provided the relevant (co)products exist and the designated classes of morphisms are closed under them.
\end{abstract}

\maketitle

\section{Introduction}\label{sec:intro}
Given a category $\mathcal A$ and a class of morphisms $S\subseteq\mathsf{Mor}(\mathcal A)$, the Gabriel--Zisman localization $Q:\mathcal A\longrightarrow\mathcal A[S^{-1}]$ turns the morphisms in $S$ into isomorphisms, and every functor that sends $S$ to isomorphisms factors uniquely through $Q$ \cite[Chapter~I, \S~1, pp.~6--7]{GabrielZisman}.
This paper studies the relationships between Gabriel--Zisman localization and products, coproducts, and additive structures, with particular attention to two questions: if $\mathcal A$ has (co)products of a certain kind, does the localized category $\mathcal A[S^{-1}]$ still have them; if both $\mathcal{A}$ and $\mathcal A[S^{-1}]$ have (co)products of a certain kind, does the localization functor preserve them?

\begin{remark}
  All indexing sets and indexing categories in this paper are assumed to be small.
  If sizes need to be recorded, one may fix Grothendieck universes $\mathcal U\in\mathcal V$, use ``small'' to mean $\mathcal U$-small, and form the required categories and localizations in $\mathcal V$; universe decorations will be omitted below.
  For conventions concerning universes, see \cite[Expos\'e~I, \S~0 and Appendix, \S\S~1--5]{SGA4} and \cite[pp.~320--321, ``Conventions ensemblistes'']{CisinskiDerivable}; for size issues in localization, see \cite[Chapter~I, \S~1.3; Chapter~II, \S\S~2.1--2.2]{VerdierDerived} and \cite[pp.~425--426]{NandaUniverseFractions}.
\end{remark}

When $S$ contains all identity morphisms, $Q$ preserves finite (co)products whenever $\mathcal A$ has finite (co)products and $S$ is closed under finite (co)products.
Maltsiniotis proves the binary (co)product statement in \cite[Proposition~2.1.8, p.~75]{MaltsiniotisGrothendieck}; Tronin proves the finite-product statement in \cite[Theorem~1.3, p.~1526]{TroninProducts}, and the finite-coproduct statement follows by duality.
See Proposition~\ref{prop:finite-structure-localization} below for the precise form used here.
The proof descends the (co)product--diagonal adjunction to the localized categories, using the compatibility of localization with finite product categories.
To generalize this result to $I$-indexed (co)products,
one considers a family of nonempty categories $(\mathcal A_i)_{i\in I}$ and classes $S_i\subseteq\mathsf{Mor}(\mathcal A_i)$ containing all identity morphisms,
and therefore studies when the canonical functor is fully faithful:
\[
  \Theta:
  \bigg(\prod_{i\in I}\mathcal A_i\bigg)
  \bigg[\bigg(\prod_{i\in I}S_i\bigg)^{-1}\bigg]
  \longrightarrow
  \prod_{i\in I}\mathcal A_i[S_i^{-1}].
\]
It is known that $\Theta$ is an isomorphism of categories when $I$ is finite; see \cite[Lemma~2.1.7, pp.~74--75]{MaltsiniotisGrothendieck}, \cite[Theorem~1.2, pp.~1524--1525]{TroninProducts}, and Lemma~\ref{lem:localization-binary-product}.
For arbitrary $I$, $\Theta$ remains bijective on objects, but it need not be full or faithful.

\medskip

To state the main theorems, the notation $\ell$ and $d$ are introduced.
For every $i\in I$, let $\overline{S_i}$ denote the class of formal inverse arrows of the arrows in $S_i$, and let $\langle\mathsf{Mor}(\mathcal A_i)\sqcup\overline{S_i}\rangle$ denote the collection of all finite compositions of arrows in $\mathsf{Mor}(\mathcal A_i)\sqcup\overline{S_i}$.
For a composition of arrows $\mathfrak P$ in this collection, let $|\mathfrak P|$ be the number of arrows in $\mathfrak P$, let $[\mathfrak P]$ be the morphism that it represents in the localization, and put
\[
  \ell(\mathfrak P)
  =\min\bigl\{|\mathfrak P'|\bigm|[\mathfrak P']=[\mathfrak P]\bigr\}.
\]
For $\mathfrak P,\mathfrak Q \in \langle\mathsf{Mor}(\mathcal A_i)\sqcup\overline{S_i}\rangle$ with $[\mathfrak P]=[\mathfrak Q]$, let $\delta(\mathfrak P,\mathfrak Q)$ be the least number of elementary rewriting steps required to rewrite $\mathfrak P$ as $\mathfrak Q$; the elementary rewritings are given by the relations in the Gabriel--Zisman construction displayed in \eqref{eq:gz-arrow-composition-relations}.
For $L\in\mathbb N$, define
\[
  d_i(L)=\sup\left\{
  \delta(\mathfrak P,\mathfrak Q)\ \middle|\
  [\mathfrak P]=[\mathfrak Q],\
  |\mathfrak P|,|\mathfrak Q|\leq L
  \right\}\in\mathbb N\cup\{\infty\},
\]
where the supremum of the empty set is taken to be $0$.

\medskip

\noindent\textbf{Theorem.} \ Given a small set $I$, a family of nonempty categories $(\mathcal A_i)_{i\in I}$, and classes of morphisms $S_i\subseteq\mathsf{Mor}(\mathcal A_i)$ containing all identity morphisms, consider the canonical functor induced by localization
\[
  \Theta:
  \bigg(\prod_{i\in I}\mathcal A_i\bigg)
  \bigg[\bigg(\prod_{i\in I}S_i\bigg)^{-1}\bigg]
  \longrightarrow
  \prod_{i\in I}\mathcal A_i[S_i^{-1}].
\]
\begin{enumerate}[label=\textup{(\arabic*)}]
  \item (Theorem~\ref{thm:product-canonical-fullness}) $\Theta$ is full if and only if there exist a finite subset $F\subseteq I$ and $N\in\mathbb N$ such that $\ell(\mathfrak P)\leq N$ for every $i \notin F$ and every $\mathfrak P\in\langle\mathsf{Mor}(\mathcal A_i)\sqcup\overline{S_i}\rangle$.
  \item (Theorem~\ref{thm:product-canonical-faithfulness}) $\Theta$ is faithful if and only if, for every $L\in\mathbb N$, there exist a finite subset $F_L\subseteq I$ and $D_L\in\mathbb N$ such that $d_i(L)\leq D_L$ for every $i \notin F_L$.
\end{enumerate}

\medskip

Moreover, when $(\mathcal{A}_i, S_i)_{i \in I}$ are all identical, one obtains the following theorem.

\medskip

\noindent\textbf{Theorem.}  (Theorem~\ref{cor:fraction-calculus-preservation}) Let $\mathcal{A}$ be a category, assume that the class of morphisms $S \subseteq \mathsf{Mor}(\mathcal{A})$ contains all identity morphisms, and let $Q : \mathcal{A} \longrightarrow \mathcal{A}[S^{-1}]$ be the localization.
Suppose that the category $\mathcal{A}$ and the class of morphisms $S$ satisfy the following two conditions:
\begin{align*}
  N = \sup \big\{\ell(\mathfrak P) \mid \mathfrak P\in \langle\mathsf{Mor}(\mathcal A)\sqcup\overline{S}\rangle\big\} < \infty, \\
  \sup\big\{ \delta(\mathfrak P,\mathfrak Q) \mid [\mathfrak P]=[\mathfrak Q], \ \mathfrak P, \mathfrak Q\in \langle\mathsf{Mor}(\mathcal A)\sqcup\overline{S}\rangle, \ |\mathfrak P|, |\mathfrak Q| \leq N+1\big\} < \infty.
\end{align*}
Then the following statements hold:
\begin{enumerate}[label=\textup{(\arabic*)}]
  \item If $\mathcal{A}$ has $I$-indexed coproducts and $S$ is closed under $I$-indexed coproducts, then $\mathcal{A}[S^{-1}]$ also has $I$-indexed coproducts and $Q$ preserves them;
  \item If $\mathcal{A}$ has $I$-indexed products and $S$ is closed under $I$-indexed products, then $\mathcal{A}[S^{-1}]$ also has $I$-indexed products and $Q$ preserves them.
\end{enumerate}

\medskip

In particular, localizations arising from a calculus of fractions or a model structure satisfy the two conditions in Theorem~\ref{cor:fraction-calculus-preservation}.
See Corollary~\ref{thm:fraction-calculus-application} and Theorem~\ref{thm:model-category-application} for the corresponding results.
The coproduct statement under a calculus of left fractions was proved in \cite[Proposition~3.5.1]{KrauseLocalizationTheory}.
For ordinary model categories, the equivalence of the comparison functor for discrete diagram categories is also recorded in \cite[Section~2.5, equation~(2.7) and footnote~14, pp.~41--42]{RiehlHomotopicalCategories}.
The proofs below recover these conclusions from the two quantitative criteria above and use Egger's weakened model-category axioms, which require only finite products and finite coproducts.

\medskip

The paper is organized as follows.
Section~\ref{sec:preliminaries} recalls basic properties of the Gabriel--Zisman localization $Q : \mathcal{A} \longrightarrow \mathcal{A}[S^{-1}]$ and sufficient conditions under which a localization functor preserves finite (co)products and direct sums.
Section~\ref{sec:finite-counterexamples} gives examples in which $\mathcal{A}$ has finite (co)products but $\mathcal{A}[S^{-1}]$ does not, and examples in which both $\mathcal{A}$ and $\mathcal{A}[S^{-1}]$ have finite (co)products but $Q$ does not preserve them.
Under the assumptions that $\mathcal{A}$ and $\mathcal{A}[S^{-1}]$ are additive and that $Q$ is additive, Section~\ref{sec:additive-infinite-counterexamples} gives, respectively, an example in which $\mathcal{A}$ has arbitrary (co)products but $\mathcal{A}[S^{-1}]$ does not, and an example in which both $\mathcal{A}$ and $\mathcal{A}[S^{-1}]$ have arbitrary (co)products but $Q$ does not preserve them.
Section~\ref{sec:product-localization} establishes the necessary and sufficient conditions for $\Theta$ to be full and faithful, together with the corresponding preservation theorem.
Section~\ref{sec:discrete-applications} applies these results to calculus of fractions and model structures.
Section~\ref{sec:conjectures-and-problems} discusses the additivity of localization functors and the localization of general functor categories, and poses several questions.

\section{Preliminaries}\label{sec:preliminaries}

\subsection{Categorical notation}

Let $\mathcal{A}$ be a category.
An object $\bot$ is called an initial object of $\mathcal{A}$ if, for every $X \in \mathsf{Ob}(\mathcal{A})$, the class of morphisms $\mathrm{Hom}_\mathcal{A}(\bot, X)$ is a singleton; dually, an object $\top$ is called a terminal object of $\mathcal{A}$ if, for every $X \in \mathsf{Ob}(\mathcal{A})$, the class of morphisms $\mathrm{Hom}_\mathcal{A}(X, \top)$ is a singleton.
Fix an indexing set $I$.
A coproduct of a family of objects $(X_i)_{i \in I}$, if it exists, consists of an object $\coprod_{i \in I} X_i$ and a family of structure morphisms $e_i : X_i \longrightarrow \coprod_{i \in I} X_i$ such that, for every $Y \in \mathsf{Ob}(\mathcal{A})$, there is an isomorphism
\begin{equation}\label{eq_coprod}
  \mathrm{Hom}_\mathcal{A}\Big(\coprod_{i \in I} X_i, Y\Big) \xlongrightarrow{\sim} \prod_{i \in I} \mathrm{Hom}_\mathcal{A}(X_i, Y),\qquad f \longmapsto (f \circ e_i)_{i \in I}.
\end{equation}
Similarly, a product of $(X_i)_{i \in I}$, if it exists, consists of an object $\prod_{i \in I} X_i$ and a family of structure morphisms $p_i : \prod_{i \in I} X_i \longrightarrow  X_i$ such that, for every $Z \in \mathsf{Ob}(\mathcal{A})$, there is an isomorphism
\begin{equation}\label{eq_prod}
  \mathrm{Hom}_\mathcal{A}\Big(Z, \prod_{i \in I} X_i\Big) \xlongrightarrow{\sim} \prod_{i \in I} \mathrm{Hom}_\mathcal{A}(Z, X_i),\qquad g \longmapsto (p_i \circ g)_{i \in I}.
\end{equation}
The isomorphisms \eqref{eq_coprod} and \eqref{eq_prod} are also called the universal properties of the coproduct and product, respectively.
By convention, $\coprod_{i \in \varnothing} X_i$ is an initial object and $\prod_{i \in \varnothing} X_i$ is a terminal object.

\medskip

A category $\mathcal{A}$ is called preadditive if, for every $X,Y,Z \in \mathsf{Ob}(\mathcal{A})$, $\mathrm{Hom}_\mathcal{A}(X,Y)$ is an abelian group and composition is additive in both variables:
\[
  \mathrm{Hom}_\mathcal{A}(Y,Z) \times \mathrm{Hom}_\mathcal{A}(X,Y) \longrightarrow \mathrm{Hom}_\mathcal{A}(X,Z),\qquad (g,f) \longmapsto g \circ f.
\]
A category $\mathcal{A}$ is called additive if it is preadditive and has finite biproducts.
Equivalently, $\mathcal A$ is preadditive, finite coproducts and finite products exist, and, for every finite family of objects $(X_i)_{i \in I}$, the canonical morphism $\coprod_{i \in I}X_i\longrightarrow\prod_{i \in I}X_i$ is an isomorphism.
The finitely indexed coproduct and product are then denoted by $\bigoplus$.

Let $\mathcal{A}$ and $\mathcal{B}$ be additive categories.
A functor $F : \mathcal{A} \longrightarrow \mathcal{B}$ is called additive if, for every $X,Y \in \mathsf{Ob}(\mathcal{A})$, the following map is a homomorphism of additive groups:
\[
  \mathrm{Hom}_\mathcal{A}(X,Y) \longrightarrow \mathrm{Hom}_\mathcal{B}(FX,FY),\qquad f \longmapsto Ff.
\]
Equivalently, for every $X,Y \in \mathsf{Ob}(\mathcal{A})$, the universal properties of the biproduct in \eqref{eq_coprod} and \eqref{eq_prod} induce an isomorphism $F(X \oplus Y) \cong FX \oplus FY$.

\subsection{Gabriel--Zisman localization and saturated classes}

Given a category $\mathcal{A}$ and a class of morphisms $S$, the construction of the Gabriel--Zisman localization is recalled from \cite[Chapter~I, \S~1, pp.~6--7]{GabrielZisman}.
Forget the identity and composition laws on $\mathsf{Mor} (\mathcal{A})$.
A morphism is then regarded merely as an arrow; the source and target of this arrow remain those of the original morphism.
Let $\overline S = \{\overline s : Y \longrightarrow X\mid s : X \longrightarrow Y\text{ belongs to }S\}$ be a family of new arrows.
Form the disjoint union of arrows $\mathsf{Mor} (\mathcal{A}) \sqcup \overline S$.
Arrows $g$ and $f$ are called composable if the source of $g$ is the target of $f$.
Write their composition as $g \ast f$.

\begin{remark}
  In particular, the composite $g \circ f$ of arrows $f,g\in\mathsf{Mor}(\mathcal A)$ gives a composition of arrows $g \ast f$ in $\mathsf{Mor} (\mathcal{A}) \sqcup \overline S$; however, $g \circ f$ is itself an arrow in $\mathsf{Mor} (\mathcal{A}) \sqcup \overline S$, and it is not equal to the composition of arrows $g \ast f$.
\end{remark}
\noindent Let $\langle \mathsf{Mor} (\mathcal{A}) \sqcup \overline S \rangle$ denote the class of all finite compositions of arrows in $\mathsf{Mor} (\mathcal{A}) \sqcup \overline S$.
Consider the equivalence relation on $\langle \mathsf{Mor} (\mathcal{A}) \sqcup \overline S \rangle$ generated by the following rules:
\begin{enumerate}[label=\textup{(\arabic*)}]
  \item if $f,g \in \mathsf{Mor}(\mathcal{A})$ and $g$ and $f$ are composable, then $g \circ f$ and $g \ast f$ are equivalent;
  \item for every arrow $a : M \longrightarrow N$, the three compositions of arrows $\operatorname{Id}_N \ast a$, $a \ast \operatorname{Id}_M$, and $a$ are equivalent;
  \item if $s : X \longrightarrow Y$ is an arrow in $S$, then $s \ast \overline s$ is equivalent to $\operatorname{Id}_Y$, and $\overline s \ast s$ is equivalent to $\operatorname{Id}_X$;
  \item if the compositions of arrows $\mathfrak P$ and $\mathfrak Q$ are equivalent, then $a \ast \mathfrak P \ast b$ and $a \ast \mathfrak Q \ast b$ are equivalent whenever these compositions of arrows are defined.
\end{enumerate}

\medskip

Let $\langle \mathsf{Mor} (\mathcal{A}) \sqcup \overline S \rangle /_\sim$ be the quotient of $\langle \mathsf{Mor} (\mathcal{A}) \sqcup \overline S \rangle$ by this equivalence relation.
One obtains the following description of $\mathcal{A}[S^{-1}]$: its class of objects is $\mathsf{Ob}(\mathcal{A})$; the class of morphisms $\mathrm{Hom}_{\mathcal{A}[S^{-1}]}(X, Y)$ is the subclass of $\langle \mathsf{Mor} (\mathcal{A}) \sqcup \overline S \rangle/_\sim$ consisting precisely of compositions of arrows from $X$ to $Y$; the identity morphism of $X \in \mathsf{Ob}(\mathcal{A}[S^{-1}])$ is induced by the arrow $\operatorname{Id}_X$; composition of morphisms is induced by $\ast$.

\medskip

By construction, a morphism in $\mathcal{A}[S^{-1}]$ is written as $[\mathfrak P]$, where $\mathfrak P \in \langle \mathsf{Mor} (\mathcal{A}) \sqcup \overline S \rangle$ is a composition of arrows and $[-]$ denotes the equivalence class under the relation above.

\subsection{Basic facts about localization}

The following basic facts about localization will be used below.

\begin{lemma}\label{lem:nullary-preservation}
  Let $\mathcal{A}$ be a category and $S \subseteq \mathsf{Mor}(\mathcal{A})$ a class of morphisms.
  Let $Q:\mathcal A\longrightarrow \mathcal A[S^{-1}]$ be the localization.
  \begin{enumerate}[label=\textup{(\arabic*)}]
    \item if $\mathcal{A}$ has an initial object $\bot$ of $\mathcal A$, then $Q\bot$ is an initial object of $\mathcal A[S^{-1}]$;
    \item if $\mathcal{A}$ has a terminal object $\top$ of $\mathcal A$, then $Q\top$ is a terminal object of $\mathcal A[S^{-1}]$;
    \item if $\mathcal{A}$ has a zero object $0$ of $\mathcal A$, then $Q0$ is a zero object of $\mathcal A[S^{-1}]$.
  \end{enumerate}
\end{lemma}

\begin{proof}
  Suppose that $\bot$ is an initial object.
  For every $A\in\mathsf{Ob}(\mathcal A)$, let $u_A:\bot\longrightarrow A$ denote the unique morphism.
  First, one proves that $\mathrm{Hom}_{\mathcal A[S^{-1}]}(Q\bot,QA)$ is a singleton.
  Take any $[\mathfrak P] \in \mathrm{Hom}_{\mathcal A[S^{-1}]}(Q\bot,QA)$, where $\mathfrak P \in \langle \mathsf{Mor} (\mathcal{A}) \sqcup \overline S \rangle$.
  One may assume that $S$ contains all identity morphisms.
  By formally inserting identity arrows, $\mathfrak P$ is equivalent to a composition of arrows of the form
  \[
    \mathfrak P' = f_n \ast \overline {s_n} \ast f_{n-1} \ast \overline{s_{n-1}} \ast \cdots \ast f_1 \ast \overline{s_1} \ast f_0,
  \]
  where $f_0$ is a morphism from the initial object, written as $f_0 = u_X : \bot \longrightarrow X$.
  For $s_1 : Y \longrightarrow X$ and $f_1 : Y \longrightarrow Z$, one has $s_1 \circ u_Y = u_X$ and $f_1 \circ u_Y = u_Z$.
  Hence $[f_1 \ast \overline{s_1} \ast f_0] = [u_Z]$.
  Inductively, $\mathfrak P'$ is equivalent to $u_A$.
  Thus $\mathrm{Hom}_{\mathcal A[S^{-1}]}(Q\bot,QA)$ is a singleton, so $Q \bot$ is initial.

  \medskip

  The argument for a terminal object is similar.
  Consequently, localization also preserves zero objects.
\end{proof}

\begin{lemma}\label{lem_loc_uni}
  Let $\mathcal{A}$ be a category and $S \subseteq \mathsf{Mor}(\mathcal{A})$ a class of morphisms.
  Then localization is unique up to an isomorphism of categories, rather than merely an equivalence of categories.
  Equivalently, if $Q_1:\mathcal A\longrightarrow \mathcal{L}_1$ and $Q_2:\mathcal A\longrightarrow \mathcal{L}_2$ are both localizations defined by the universal property, then there is an isomorphism $\mathcal{L}_1 \cong \mathcal{L}_2$.
\end{lemma}

\begin{proof}
  Let $\widetilde {Q_2} : \mathcal{L}_1 \longrightarrow \mathcal{L}_2$ be the functor induced by the universal property of the localization $Q_1$, and let $\overline{Q_1} : \mathcal{L}_2 \longrightarrow \mathcal{L}_1$ be the functor induced by the universal property of $Q_2$.
  By construction, $Q_1 = \overline{Q_1} \circ \widetilde{Q_2} \circ Q_1$ and $Q_2 = \widetilde{Q_2} \circ \overline{Q_1} \circ Q_2$.
  The universal property of localization gives $\operatorname{Id}_{\mathcal{L}_1} = \overline{Q_1} \circ \widetilde{Q_2}$ and $\operatorname{Id}_{\mathcal{L}_2} = \widetilde{Q_2} \circ \overline{Q_1}$.
  This constructs an isomorphism between the categories $\mathcal{L}_1$ and $\mathcal{L}_2$.
\end{proof}

\begin{lemma}\label{lem:localization-binary-product}
  Let $I$ be a finite indexing set.
  For every $i \in I$, let $\mathcal{A}_i$ be a category and $S_i \subseteq \mathsf{Mor}(\mathcal{A}_i)$ a class of morphisms.
  Assume that $S_i$ contains all identity morphisms.
  Let $Q_i : \mathcal{A}_i \longrightarrow \mathcal{A}_i[S_i^{-1}]$ be the localization.
  Then the following functor induced by the universal property of localization is an isomorphism of categories:
  \[
    \widetilde {\prod_{i \in I} Q_i} :
    \bigg(\prod_{i \in I} \mathcal{A}_i \bigg)\bigg[\big(\prod_{i \in I}S_i\big)^{-1}\bigg]
    \longrightarrow\prod_{i \in I}\mathcal{A}_i[S_i^{-1}].
  \]
\end{lemma}

\begin{proof}
  See \cite[Lemma~2.1.7, pp.~74--75]{MaltsiniotisGrothendieck}.
\end{proof}

The binary product and coproduct assertions in the following proposition are proved in \cite[Proposition~2.1.8, p.~75]{MaltsiniotisGrothendieck}.
The product assertion also follows from \cite[Theorem~1.3, p.~1526]{TroninProducts}, while \cite[Theorem~3.1, pp.~1535--1536]{TroninProducts} gives the corresponding preadditive conclusion.
Since the descent of the coproduct--diagonal adjunction in the finite case provides the basic idea for the main theorems below, the details are included here.

\begin{proposition}\label{prop:finite-structure-localization}
  Let $\mathcal{A}$ be a category and $S \subseteq \mathsf{Mor}(\mathcal{A})$ a class of morphisms.
  Let $Q:\mathcal A\longrightarrow \mathcal A[S^{-1}]$ be the localization.
  Assume further that $S$ contains all identity morphisms: this does not change the localization, but makes the statement more convenient.
  \begin{enumerate}[label=\textup{(\arabic*)}]
    \item If $\mathcal{A}$ has finite coproducts and $S$ is closed under finite coproducts, then $\mathcal{A}[S^{-1}]$ has finite coproducts and $Q$ preserves them;
    \item If $\mathcal{A}$ has finite products and $S$ is closed under finite products, then $\mathcal{A}[S^{-1}]$ has finite products and $Q$ preserves them;
    \item If $\mathcal{A}$ is an additive category and $S$ is closed under finite biproducts, then $\mathcal{A}[S^{-1}]$ is an additive category and $Q$ preserves finite biproducts.
  \end{enumerate}
\end{proposition}

\begin{proof}
  It suffices first to prove (1); assertion (2) is dual.
  (3) follows from (1) and (2).

  \medskip

  By Lemma~\ref{lem:nullary-preservation}, it remains only to prove the following: if $\mathcal{A}$ has finite coproducts and $S$ is closed under finite coproducts, then $\mathcal{A}[S^{-1}]$ has binary coproducts and $Q$ preserves them.
  The binary coproduct functor $\coprod : \mathcal{A} \times \mathcal{A} \longrightarrow \mathcal{A}$ and the diagonal functor $\Delta : \mathcal{A} \longrightarrow \mathcal{A} \times \mathcal{A}$ are respectively left and right adjoint.
  In particular, the right adjoint $\Delta$ sends $S$ to $S \times S$; since $S$ is closed under binary coproducts, the left adjoint $\coprod$ sends $S \times S$ to $S$.
  By \cite[Lemma~1.1.6, p.~9]{KrauseHomologicalTheory}, or by the primary source \cite[Chapter~I, \S~1, Lemma~1.2, p.~7]{GabrielZisman}, the adjunction $\coprod \dashv \Delta$ therefore induces an adjunction on the localized categories
  \[
    \xymatrix@C=30mm@R=10mm{
    \mathcal A\times\mathcal A
    \ar@<0.8ex>[r]^{\coprod}
    \ar[d]_{Q\times Q}
    &
    \mathcal A
    \ar@<0.8ex>[l]^{\Delta}
    \ar[d]^{Q}
    \\
    \mathcal A[S^{-1}]\times\mathcal A[S^{-1}]
    \ar@<0.8ex>[r]^-{\coprod_S}
    &
    \mathcal A[S^{-1}]
    \ar@<0.8ex>[l]^-{\Delta_S}
    } \qquad \xymatrix@R=10mm{
    Q \circ \coprod = \coprod _S \circ (Q \times Q),\\
    (Q \times Q) \circ \Delta = \Delta_S \circ Q.
    }
  \]
  By Lemma~\ref{lem:localization-binary-product}, one may regard $Q \times Q$ as the localization functor of $\mathcal{A} \times \mathcal{A}$ at $S \times S$.
  It is clear that $\Delta_S$ is still the diagonal functor and has a left adjoint.
  Consequently, $\mathcal{A}[S^{-1}]$ has binary coproducts and hence finite coproducts.
  The commutative diagram shows that the localization $Q$ preserves coproducts.
\end{proof}

\section{Localization and (co)products}\label{sec:finite-counterexamples}

\subsection{Examples in which \texorpdfstring{$\mathcal{A}$}{} has finite (co)products but \texorpdfstring{$\mathcal{A}[S^{-1}]$}{} does not}

Consider a localization $Q:\mathcal A\longrightarrow\mathcal A[S^{-1}]$.
The following example shows that even if $\mathcal{A}$ has finite (co)products, $\mathcal{A}[S^{-1}]$ need not have finite (co)products.

\begin{example}\label{ex_seven_element_lattice}
  Let $\mathcal A$ be the poset category determined by the following Hasse diagram:
  \[
    \vcenter{\hbox{$
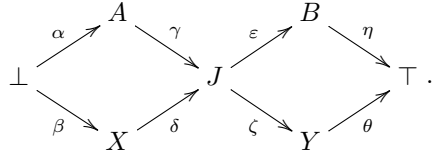
\xymatrix@C=8mm@R=4mm{
      &A\ar[dr]^{\gamma}&&B\ar[dr]^{\eta}&\\
      \bot\ar[ur]^{\alpha}\ar[dr]_{\beta}&&
      J\ar[ur]^{\varepsilon}\ar[dr]_{\zeta}&&\top\\
      &X\ar[ur]_{\delta}&&Y\ar[ur]_{\theta}&
      }$}}.
  \]
  This is a finite lattice, so $\mathcal A$ has all finite products and finite coproducts.
  Let $S=\{s\} = \{\varepsilon\circ\gamma\}$.

  \medskip

  It will be shown that $QX \coprod QX$ and $QY \prod QY$ do not exist in $\mathcal A[S^{-1}]$.
\end{example}

\begin{proof}
  Take any morphism $[\mathfrak P]$ in $\mathcal{A}[S^{-1}]$, where $\mathfrak P \in \langle\mathsf{Mor}(\mathcal A)\sqcup\overline S\rangle$ is a composition of arrows.
  If two arrows from $\overline {S}$ occur in $\mathfrak P$, then between them there is an arrow or a composition of arrows from $A$ to $B$; hence $\mathfrak P$ has the form
  \[
    \cdots \ast \overline s \ast s \ast \overline s \ast \cdots \quad \text{or} \quad \cdots \ast \overline s \ast \varepsilon \ast \gamma \ast \overline s \ast \cdots.
  \]
  Since $\overline s \ast s \ast \overline s$ and $\overline s \ast \varepsilon \ast \gamma \ast \overline s$ are both equivalent to $\overline s$, one may assume that $\overline s$ occurs at most once in $\mathfrak P$.
  Therefore, for every $M, N \in \mathcal{A}$,
  \begin{equation}\label{eq_hom_bound}
    |\mathrm{Hom}_{\mathcal{A}[S^{-1}]}(QM,QN)|
    \leq \underset{\substack{\text{number of morphisms}\\\text{containing no }\overline s}}{\underbracket{|\mathrm{Hom}_{\mathcal{A}}(M,N)|}}
    \quad + \underset{\substack{\text{number of morphisms containing}\\\text{exactly one }\overline s}}{\underbracket{|\mathrm{Hom}_{\mathcal{A}}(M,B)| \cdot |\mathrm{Hom}_{\mathcal{A}}(A,N)|}}
    \leq 2.
  \end{equation}
  Next, one proves that $\mathrm{Hom}_{\mathcal A[S^{-1}]}(QX,QY)$ contains two distinct morphisms.
  Together with \eqref{eq_hom_bound}, this will show that its cardinality is $2$.
  Let $k$ be a field.
  Consider the functor $F : \mathcal{A} \to \mathrm{Vect}_k$ given by
  \[
    F :\quad
    \vcenter{\hbox{$\xymatrix@C=6mm@R=4mm{
      &A\ar[dr]^{\gamma}&&B\ar[dr]^{\eta}&\\
      \bot\ar[ur]^{\alpha}\ar[dr]_{\beta}&&
      J\ar[ur]^{\varepsilon}\ar[dr]_{\zeta}&&\top\\
      &X\ar[ur]_{\delta}&&Y\ar[ur]_{\theta}&
      }$}} \quad
    \longrightarrow \quad
    \vcenter{\hbox{$\xymatrix@C=6mm@R=4mm{
      &k\ar[dr]^{e_1}&&k\ar[dr]^{0}&\\
      0\ar[ur]^{0}\ar[dr]_{0}&&
      k^2\ar[ur]^{p_1}\ar[dr]_{p_2}&&0\\
      &k\ar[ur]_{e_2}&&k\ar[ur]_{0}&
      }$}}.
  \]
  By construction, $Fs = \operatorname{Id}_k$.
  The universal property of localization gives a unique functor
  \[
    \widetilde F:\mathcal A[S^{-1}]
    \longrightarrow\operatorname{Vect}_k,
    \qquad \widetilde F\circ Q=F.
  \]
  The morphisms $[\zeta\ast\delta]$ and $[\zeta\ast\gamma\ast\overline s\ast \varepsilon\ast\delta]$ are distinct in $\mathrm{Hom}_{\mathcal{A}[S^{-1}]}(QX, QY)$.
  By construction,
  \begin{align*}
    \widetilde F([\zeta\ast\delta])
     & =F(\zeta) \circ F(\delta)
    =p_2\circ e_2
    =\operatorname{Id}_k,                                                             \\
    \widetilde F([\zeta\ast\gamma\ast\overline s\ast \varepsilon\ast\delta])
     & =F(\zeta) \circ F(\gamma) \circ F(s)^{-1} \circ F(\varepsilon) \circ F(\delta)
    =p_2\circ e_1\circ \operatorname{Id}_k^{-1}\circ p_1\circ e_2
    =0.
  \end{align*}
  Therefore $\widetilde F([\zeta\ast\delta]) \neq \widetilde F([\zeta\ast\gamma\ast\overline s\ast \varepsilon\ast\delta])$, and hence $[\zeta\ast\delta] \neq [\zeta\ast\gamma\ast\overline s\ast \varepsilon\ast\delta]$.
  This proves that $\left| \mathrm{Hom}_{\mathcal A[S^{-1}]}(QX,QY) \right|=2$.

  \medskip

  It remains to show that $QX \coprod QX$ does not exist.
  Otherwise, \eqref{eq_coprod} would give
  \[
    \big|\mathrm{Hom}_{\mathcal A[S^{-1}]}(QX\coprod QX,QY)\big| = \big|\mathrm{Hom}_{\mathcal A[S^{-1}]}(QX,QY)\big|^2 = 4.
  \]
  This contradicts \eqref{eq_hom_bound}.
  If $QY\prod QY$ existed, then \eqref{eq_prod} would similarly give
  \[
    \big|\mathrm{Hom}_{\mathcal A[S^{-1}]}(QX,QY\prod QY)\big|
    =\big|\mathrm{Hom}_{\mathcal A[S^{-1}]}(QX,QY)\big|^2=4,
  \]
  again contradicting \eqref{eq_hom_bound}.
\end{proof}

The following example shows that a Gabriel--Zisman localization of an additive category need not be additive.
A complete proof valid over an arbitrary field is given below; in fact, the resulting localization is not even preadditive.

\begin{example}\label{ex:vector-space-nonpreadditive-localization}
  Let $\mathcal A$ be the category of finite-dimensional $k$-vector spaces, let $S = \{0 \xlongrightarrow s k\}$, and let $Q : \mathcal{A} \longrightarrow \mathcal{A}[S^{-1}]$ be the localization.
  Then $\mathcal A[S^{-1}]$ is not an additive category.
\end{example}

\begin{proof}
  By Lemma~\ref{lem:nullary-preservation}, $Q0$ is a zero object of $\mathcal A[S^{-1}]$.
  Since $Qs:Q0\longrightarrow Qk$ is an isomorphism, $Qk$ is also a zero object.
  Moreover $Q$ is full.
  Indeed, $Q(s)^{-1}$ is represented by $k \longrightarrow 0$.
  It follows that for each $X,Y \in \mathcal{A}$, $\mathrm{Hom}_{\mathcal{A}[S^{-1}]}(QX,QY)$ is a quotient of $\mathrm{Hom}_{\mathcal{A}}(X,Y)$.

  \medskip

  Consider the functor
  \[
    \Phi:\mathcal A\longrightarrow\mathcal A, \qquad
    \Phi(V)=\bigwedge\nolimits^2V \oplus\bigl(\bigwedge\nolimits^2V\bigr)\otimes_kV,\quad
    \Phi(f)=\bigwedge\nolimits^2f \oplus\bigl(\bigwedge\nolimits^2f\bigr)\otimes_kf,
  \]
  Since both $\Phi k$ and $\Phi 0$ are zero objects, $\Phi s$ is an isomorphism.
  The universal property of localization gives a functor $\widetilde\Phi:\mathcal{A}[S^{-1}]\longrightarrow\mathcal A$ satisfying $\widetilde\Phi\circ Q=\Phi$.
  Fix an isomorphism $\bigwedge^2k^2\cong k$.
  For every $g\in\operatorname{GL}_2(k)$, a direct calculation gives $\Phi(g)=\det(g)\oplus\bigl(\det(g) \cdot g\bigr)$.
  Therefore, if $g,h\in\operatorname{GL}_2(k)$ and $[g]=[h]$, comparing the first components in this formula gives $\det(g)=\det(h)$, and then comparing the second components gives $g=h$.
  It follows that
  \[
    \operatorname{GL}_2(k)\longrightarrow
    \operatorname{End}_{\mathcal A [S^{-1}]}(Qk^2),
    \qquad g\longmapsto[g],
  \]
  is injective.
  Moreover, $\Phi(\operatorname{Id}_{k^2})\neq\Phi(0)$, so $[\operatorname{Id}_{k^2}]\neq[0]$.

  \medskip

  Now suppose, for a contradiction, that $\mathcal A [S^{-1}]$ is preadditive.
  Since $Qk$ is a zero object, every rank $1$ linear map $f:k^2\longrightarrow k^2$ satisfies $[f] = [0]$.
  On the other hand, every rank $2$ linear map is an isomorphism.
  Since $Q$ is full, every element of $\operatorname{End}_{\mathcal A [S^{-1}]}(Qk^2)$ has the form $[f]$.
  Hence the ring $D = \operatorname{End}_{\mathcal A [S^{-1}]}(Qk^2)$ is nonzero and every nonzero element is invertible; thus $D$ is a division ring, so $1_D$ has at most two square roots.
  If $\operatorname{char}(k)\neq2$, then $\pm \begin{psmallmatrix}-1&0\\0&1\end{psmallmatrix}$ give a contradiction; if $\operatorname{char}(k)=2$, then $\begin{psmallmatrix}1&1\\0&1\end{psmallmatrix}$ and $\begin{psmallmatrix}1&0\\1&1\end{psmallmatrix}$ give a contradiction.

  \medskip

  Thus $\mathcal A[S^{-1}]$ is not preadditive and hence is not additive.
\end{proof}

\subsection{Examples in which \texorpdfstring{$\mathcal{A}$}{} and \texorpdfstring{$\mathcal{A}[S^{-1}]$}{} have finite (co)products that are not preserved by \texorpdfstring{$Q$}{}}

Consider a localization $Q:\mathcal A\longrightarrow\mathcal A[S^{-1}]$.
The following example shows that even if both $\mathcal{A}$ and $\mathcal{A}[S^{-1}]$ have finite (co)products, $Q$ need not preserve them.

\begin{example}\label{ex:diamond-lattice}
  Let $\mathcal A$ be the poset category determined by the following diamond lattice:
  
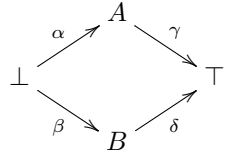
\[
    \xymatrix@C=8mm@R=4mm{
    &A\ar[dr]^{\gamma}&\\
    \bot\ar[ur]^{\alpha}\ar[dr]_{\beta}&&\top\\
    &B\ar[ur]_{\delta}&
    }
  \]
  Let $S=\{\alpha\}$ and let $Q:\mathcal A\longrightarrow\mathcal A[S^{-1}]$ be the localization.
  Both $\mathcal A$ and $\mathcal A[S^{-1}]$ have finite coproducts, but $Q$ does not preserve binary coproducts.

\end{example}

\begin{proof}
  Let $\mathbf 3=(0<1<2)$ be the poset category determined by the three-element chain.
  Define a functor $F:\mathcal A\longrightarrow\mathbf 3$ by
  \[
    F(\bot)=F(A)=0,
    \qquad F(B)=1,
    \qquad F(\top)=2.
  \]
  Then $F(\alpha)=\operatorname{Id}_0$.
  By the universal property of localization, $F$ factors through $Q$ as $\widetilde F:\mathcal A[S^{-1}]\longrightarrow\mathbf 3$, with $\widetilde F\circ Q=F$.

  \medskip

  The functor $\widetilde F$ is an equivalence of categories.
  In the other direction, define a functor $G:\mathbf 3\longrightarrow \mathcal A[S^{-1}]$.
  On objects and morphisms it is given by
  \[
    G : \quad 0 \longrightarrow 1 \longrightarrow 2 \quad \longmapsto \quad
    QA \xlongrightarrow{[\beta \ast \overline \alpha]} QB \xlongrightarrow{[\delta]} Q\top.
  \]
  Here $[\delta\ast\beta\ast\overline\alpha] = [\gamma\ast\alpha\ast\overline\alpha]= [\gamma]$.
  This shows that $G$ is indeed a functor.
  A direct calculation gives $\widetilde FG=\operatorname{Id}_{\mathbf 3}$.
  Define a natural isomorphism $\eta:\operatorname{Id}_{\mathcal A[S^{-1}]}\cong G\widetilde F$; its components are the vertical arrows in the following diagram:
  \[
    \xymatrix@C=13mm@R=8mm{
    \operatorname{Id}&
    Q\bot\ar@<.55ex>[r]^{[\alpha]}\ar[d]_{[\alpha]}&
    QA\ar@<.55ex>[l]^{[\overline\alpha]}
    \ar[r]^{[\beta\ast\overline\alpha]}
    \ar[d]^{[\operatorname{Id}_A]}&
    QB\ar[r]^{[\delta]}\ar[d]^{[\operatorname{Id}_B]}&
    Q\top\ar[d]^{[\operatorname{Id}_{\top}]}\\
    G\widetilde F&
    QA\ar@<.55ex>[r]^{[\operatorname{Id}_A]}&
    QA\ar@<.55ex>[l]^{[\operatorname{Id}_A]}
    \ar[r]_{[\beta\ast\overline\alpha]}&
    QB\ar[r]_{[\delta]}&
    Q\top.
    }
  \]
  Since $\mathbf 3$ is a bounded lattice, it has finite products and finite coproducts, and therefore so does $\mathcal{A}[S^{-1}]$.
  But $Q$ does not preserve binary coproducts.
  Indeed,
  \[
    Q (A \coprod B) = Q \top \ncong QB \cong QA \coprod QB.
  \]
  This gives an example in which both $\mathcal{A}$ and $\mathcal{A}[S^{-1}]$ have finite coproducts, but $Q$ does not preserve binary coproducts.
\end{proof}

Passing to opposite categories gives an example in which both $\mathcal{A}$ and $\mathcal{A}[S^{-1}]$ have finite products, but $Q$ does not preserve binary products.

\section{Additive localizations and infinite (co)products}\label{sec:additive-infinite-counterexamples}

\subsection{\texorpdfstring{An example in which $Q$ is additive, $\mathcal{A}$ has arbitrary coproducts, but $\mathcal{A}[S^{-1}]$ does not}{An example in which Q is additive, A has arbitrary coproducts, but its localization does not}}

Let $Q : \mathcal{A} \longrightarrow \mathcal{A}[S^{-1}]$ be a localization.
Assume that $\mathcal{A}$ and $\mathcal{A}[S^{-1}]$ are additive categories and that $Q$ is an additive functor.
Even if $\mathcal{A}$ has arbitrary coproducts, $\mathcal{A}[S^{-1}]$ need not have arbitrary coproducts.

\begin{example}\label{ex:additive-localization-without-countable-coproduct}
  Let $\mathcal A$ be the category of free abelian groups; its objects have the form $\coprod_{i\in I}\mathbb Z$, where $I$ is a set, and its morphisms are group homomorphisms.
  Let
  \[
    S=\left\{s:F\longrightarrow G\ \middle|\
    s\text{ is injective, and there exists $n\geq0$ such that }2^n\operatorname{coker}(s)=0\right\}.
  \]
  Then $\mathcal A$ has coproducts, $\mathcal A[S^{-1}]$ is an additive category, and the localization $Q:\mathcal A\longrightarrow\mathcal A[S^{-1}]$ is additive; however, the coproduct of the countable family $(Q\mathbb Z)_{n\geq1}$ does not exist in $\mathcal A[S^{-1}]$.
\end{example}

\begin{proof}
  It is clear that $S$ is closed under finite biproducts.
  By Proposition~\ref{prop:finite-structure-localization}(3), $\mathcal A[S^{-1}]$ is an additive category and $Q$ is an additive functor.

  \medskip

  Let $R=\mathbb Z[1/2]$ and consider the tensor functor
  \[
    T=-\otimes_{\mathbb Z}R:
    \mathcal A\longrightarrow R\text{-}\mathsf{Mod}.
  \]
  Since $R$ is a flat $\mathbb Z$-module, the tensor functor is exact.
  For every $S$-morphism $s:F\longrightarrow G$, there is a short exact sequence
  \[
    0\longrightarrow F\otimes_{\mathbb Z}R
    \xlongrightarrow{T(s)}G\otimes_{\mathbb Z}R
    \longrightarrow\operatorname{coker}(s)\otimes_{\mathbb Z}R
    \longrightarrow0.
  \]
  Since $2$ is invertible in $R$ and $\operatorname{coker}(s)$ is annihilated by a power of $2$, one has $\operatorname{coker}(s)\otimes_{\mathbb Z}R=0$, so $T(s)$ is an isomorphism.
  By the universal property of localization, there is a unique functor $\widetilde T:\mathcal A[S^{-1}]\longrightarrow R\text{-}\mathsf{Mod}$ such that $T=\widetilde T\circ Q$.

  \medskip

  It remains to prove that the coproduct of $(Q\mathbb Z)_{n\geq1}$ does not exist.
  Suppose, for a contradiction, that an object $QC$ together with structure morphisms $(e_n:Q\mathbb Z\longrightarrow QC)_{n \geq 1}$ forms this coproduct.
  For each $n\geq1$, let $x_n=\widetilde T(e_n)(1)\in C\otimes_{\mathbb Z}R$.
  Regard $C$ as a subgroup of $C\otimes_{\mathbb Z}R = \bigcup_{d\geq0}2^{-d}C$.
  One may therefore choose $d_n\geq0$ such that $2^{d_n}x_n\in C$.
  For every $m\geq0$, define the multiplication map
  \[
    \mu_m:\mathbb Z\longrightarrow\mathbb Z,
    \qquad a\longmapsto2^ma.
  \]
  By construction, $\mu_m\in S$.
  Thus $[\overline{\mu_m}] : Q\mathbb Z \longrightarrow Q\mathbb Z$ is a well-defined morphism in $\mathcal{A}[S^{-1}]$.
  Applying the universal property of the coproduct to the family of morphisms $([\overline{\mu_{d_n+n}}])_{n\geq1}$, there is a unique morphism $u:QC\longrightarrow Q\mathbb Z$ such that $u\circ e_n=[\overline{\mu_{d_n+n}}]$ for all $n \geq 1$.

  Choose a finite arrow representation $u=[\mathfrak P]$.
  Suppose that all inverse arrows occurring in $\mathfrak P$ are $(\overline{s_i}:G_i\longrightarrow F_i)_{1 \leq i \leq r}$, where $s_i:F_i\longrightarrow G_i$ belongs to $S$.
  For every $i$, choose $N_i\geq0$ such that $2^{N_i}\operatorname{coker}(s_i)=0$.
  Then $T(s_i)^{-1}(G_i)\subseteq2^{-N_i}F_i$.
  Indeed, for every $y\in G_i$, one can choose $z\in F_i$ such that $s_i(z)=2^{N_i}y$, and hence $T(s_i)^{-1}(y)=2^{-N_i}z$.
  An ordinary arrow does not increase the denominator, whereas an inverse arrow $\overline{s_i}$ multiplies the denominator by at most $2^{N_i}$.
  Composing successively along the arrows in $\mathfrak P$ therefore gives
  \[
    \widetilde T(u)(C)\subseteq2^{-N}\mathbb Z , \qquad (N =N_1+\cdots+N_r).
  \]
  If no inverse arrow occurs in $\mathfrak P$, take $N=0$ in this formula.
  Since $2^{d_n}x_n\in C$, one has
  \[
    2^{d_n}\widetilde T(u)(x_n)
    =\widetilde T(u)(2^{d_n}x_n)
    \in2^{-N}\mathbb Z.
  \]
  On the other hand, $u\circ e_n=[\overline{\mu_{d_n+n}}]$ gives
  \[
    \widetilde T(u)(x_n)
    =\widetilde T(u\circ e_n)(1)
    =\widetilde T([\overline{\mu_{d_n+n}}])(1)
    =2^{-(d_n+n)}.
  \]
  Choose $n>N$.
  Then $2^{-(d_n+n)}\notin2^{-(d_n+N)}\mathbb Z$, a contradiction.
\end{proof}

Passing to opposite categories in the preceding example shows that $\mathcal{A}$ and $\mathcal{A}[S^{-1}]$ may both be additive categories and $Q$ additive, while $\mathcal{A}$ has arbitrary products but $\mathcal{A}[S^{-1}]$ does not.

\subsection{\texorpdfstring{An example in which $Q$ is additive and $\mathcal{A}$ and $\mathcal{A}[S^{-1}]$ have arbitrary products that are not preserved by $Q$}{An example in which Q is additive and A and its localization have arbitrary products that are not preserved by Q}}

\begin{example}\label{ex:torsion-complexes}
  Let $\mathsf{TorsAb}$ denote the category of torsion abelian groups.
  An object $G \in \mathsf{Ab}$ is called a torsion abelian group if every $g \in G$ has finite order.
  Let $\mathcal A=\operatorname{Ch}(\mathsf{TorsAb})$ be the category of cochain complexes in $\mathsf{TorsAb}$.
  Let $S$ be the class of all quasi-isomorphisms in $\mathcal{A}$, and let $Q:\mathcal A\longrightarrow\mathcal A[S^{-1}]$ be the corresponding additive localization.
  Then both $\mathcal A$ and $\mathcal A[S^{-1}]$ have arbitrary products, but $Q$ does not preserve countable products.
\end{example}

\begin{proof}
  Let $t(G)$ denote the torsion subgroup of an abelian group $G$, namely the subgroup consisting of all elements of finite order.
  For any indexing set $I$ and any family of torsion abelian groups $(A_i)_{i\in I}$, their product in $\mathsf{TorsAb}$ is $t\left(\prod_{i\in I}A_i\right)$, where $\prod_{i\in I}A_i$ is the product in $\mathsf{Ab}$.
  Indeed, for any torsion abelian group $t(G)$ and any family of morphisms $(f_i:t(G)\longrightarrow A_i)_{i\in I}$, there is a unique morphism $f:t(G)\longrightarrow\prod_{i\in I}A_i$ such that $f_i=p_i\circ f$ for all $i\in I$, where $p_i:\prod_{i\in I}A_i\longrightarrow A_i$ is the structure morphism.
  In particular, the image of $f$ lies in $t\left(\prod_{i\in I}A_i\right)$, which shows that $t\left(\prod_{i\in I}A_i\right)$ is the product in $\mathsf{TorsAb}$.

  \medskip

  Fix a prime $p$.
  For every $n\geq1$, consider the epimorphism
  \[
    f_n:\mathbb Z/p^n\mathbb Z\longrightarrow\mathbb Z/p\mathbb Z,
    \qquad  a + p^n \mathbb Z \longmapsto a + p \mathbb Z .
  \]
  The family of morphisms $(f_n)_{n\geq1}$ induces the following epimorphism in $\mathsf{Ab}$:
  \[
    \varphi_0 : \prod\nolimits_{n\geq1}^{\mathsf{Ab}}\mathbb Z/p^n\mathbb Z
    \longrightarrow
    \prod\nolimits_{n\geq1}^{\mathsf{Ab}}\mathbb Z/p\mathbb Z,
    \qquad
    (x_n)_{n\geq1}\longmapsto(f_n(x_n))_{n\geq1}.
  \]
  Let $\varphi$ be the restriction of $\varphi_0$ to the torsion subgroup, namely
  \[
    \varphi:
    t\left(\prod_{n\geq1}\mathbb Z/p^n\mathbb Z\right)
    \longrightarrow
    \prod_{n\geq1}\mathbb Z/p\mathbb Z,
    \qquad
    (x_n)_{n\geq1}\longmapsto(f_n(x_n))_{n\geq1}.
  \]
  The morphism $\varphi$ is not surjective.
  Indeed, the element $(\overline1,\overline1,\ldots)$ is not in the image of $\varphi$.
  For any $(x_n)_{n \geq 1} \in \varphi_0 ^{-1}((\overline1,\overline1,\ldots))$, the equality $f_n(x_n)=\overline1$ implies that $x_n$ has order $p^n$.
  Hence no nonzero integer annihilates all components of $(x_n)_{n\geq1}$.
  Thus $(x_n)_{n \geq 1} \notin t\left(\prod_{n\geq1}\mathbb Z/p^n\mathbb Z\right)$, and $\varphi$ is not surjective.

  \medskip

  For every $n\geq1$, regard the short exact sequence
  \[
    0 \longrightarrow \ker(f_n) \longrightarrow \mathbb Z/p^n\mathbb Z \xlongrightarrow{f_n} \mathbb Z/p\mathbb Z \longrightarrow 0
  \]
  as an acyclic complex $K_n$ concentrated in degrees $0,1,2$.
  Let $K=\prod_{n\geq1}K_n$ be the product in $\mathcal A$.
  Since products are computed degreewise, the differential of $K$ from degree $1$ to degree $2$ is precisely $\varphi$.
  Because $\varphi$ is not surjective, $H^2(K)=\operatorname{coker}_{\mathsf{TorsAb}}(\varphi)\neq0$.
  Thus $K$ is not acyclic.

  The derived category is the localization $Q : \mathcal{A} \longrightarrow \mathcal A[S^{-1}]=D(\mathsf{TorsAb})$.
  Here $Q$ is the composite of two additive functors $\mathcal A\longrightarrow K(\mathsf{TorsAb}) \longrightarrow D(\mathsf{TorsAb})$.
  The category $\mathsf{TorsAb}$ is a Grothendieck abelian category\footnote{The category $\mathsf{TorsAb}$ is a Serre subcategory of $\mathsf{Ab}$ and is closed under arbitrary coproducts, so the assertion follows from \cite[Proposition~2.2.16 and its proof, pp.~36--37]{KrauseHomologicalTheory}. Taking the generator $\mathbb Z$ of $\mathsf{Ab}$ in the construction given there yields the generator $\coprod_{m\geq1}\mathbb Z/m\mathbb Z$ of $\mathsf{TorsAb}$.}.
  By Serp\'e's theorem on $K$-injective resolutions \cite[Theorem~3.13]{SerpeKInjective}, every unbounded complex has a $K$-injective resolution.
  The degreewise product of a family of $K$-injective complexes $(J_\lambda)_{\lambda\in\Lambda}$ is again $K$-injective: for every acyclic complex $X$, there is a canonical isomorphism
  \[
    \operatorname{Hom}^{\bullet}\!\left(X,\prod_{\lambda\in\Lambda}J_\lambda\right)
    \cong
    \prod_{\lambda\in\Lambda}\operatorname{Hom}^{\bullet}(X,J_\lambda),
  \]
  and arbitrary products of abelian groups preserve exact sequences.
  Therefore $D(\mathsf{TorsAb})$ has arbitrary products, which may be computed as degreewise products of $K$-injective resolutions.

  \medskip

  Finally, every $K_n$ is acyclic, so $\prod_{n\geq1}QK_n\cong0$.
  But $H^2(K)\neq0$, so $QK\not\cong0$.
  Hence the canonical morphism induced by the product projections
  \[
    Q\left(\prod_{n\geq1}K_n\right)
    =QK
    \longrightarrow
    \prod_{n\geq1}QK_n
  \]
  is not an isomorphism.
  This shows that $Q$ does not preserve countable products.
\end{proof}

Passing to opposite categories in the preceding example gives the following example: $\mathcal{A}$ and $\mathcal{A}[S^{-1}]$ are additive categories with arbitrary coproducts, while $Q$ is additive but does not preserve countable coproducts.

\section{Localization of product categories}\label{sec:product-localization}

The study of whether localization preserves finite (co)products used the isomorphism in Lemma~\ref{lem:localization-binary-product}.
Likewise, when studying whether localization preserves arbitrary (co)products, the main obstruction is the canonical functor
\[
  \Theta : \widetilde {\prod_{i \in I} Q_i} :
  \bigg(\prod_{i \in I} \mathcal{A}_i \bigg)\bigg[\big(\prod_{i \in I}S_i\big)^{-1}\bigg]
  \longrightarrow\prod_{i \in I}\mathcal{A}_i[S_i^{-1}],
\]
where $I$ is an arbitrary indexing set, not necessarily finite.
Henceforth every $\mathcal A_i$ is assumed to be nonempty, and $S_i$ is assumed to contain all identity morphisms of $\mathcal A_i$.
If some $\mathcal A_i$ is empty, then the source and target of $\Theta$ are both empty categories; if $I=\varnothing$, then both are terminal categories.
In either case, $\Theta$ is an isomorphism of categories.

\subsection{A necessary and sufficient condition for \texorpdfstring{$\Theta$}{Theta} to be full}
For every $i\in I$, define
\[
  \ell:
  \langle\mathsf{Mor}(\mathcal A_i)\sqcup\overline{S_i}\rangle
  \longrightarrow\mathbb N,
  \qquad
  \ell(\mathfrak P)
  =
  \min\bigl\{|\mathfrak P'|\bigm|[\mathfrak P']=[\mathfrak P]\bigr\},
\]
where $|\mathfrak P|$ denotes the number of arrows in $\mathfrak P$.

\begin{theorem}\label{thm:product-canonical-fullness}
  The canonical functor
  \[
    \Theta : \widetilde {\prod_{i \in I} Q_i} :
    \bigg(\prod_{i \in I} \mathcal{A}_i \bigg)\bigg[\big(\prod_{i \in I}S_i\big)^{-1}\bigg]
    \longrightarrow\prod_{i \in I}\mathcal{A}_i[S_i^{-1}]
  \]
  is full if and only if there exist a finite subset $F\subseteq I$ and $N\in\mathbb N$ such that $\ell(\mathfrak P)\leq N$ for every $i\notin F$ and every $\mathfrak P\in \langle\mathsf{Mor}(\mathcal A_i)\sqcup\overline{S_i}\rangle$.
\end{theorem}

\begin{proof}
  Suppose first that $\Theta$ is full.
  Assume for contradiction that the stated condition fails, i.e.,
  \begin{equation}\label{eq:fullness-bound-negation}
    \forall E\subseteq I\text{ finite},\quad
    \forall n\in\mathbb N,\quad
    \exists i\in I\setminus E,\quad
    \exists\mathfrak P_i\in
    \langle\mathsf{Mor}(\mathcal A_i)\sqcup\overline{S_i}\rangle,
    \quad
    \ell(\mathfrak P_i)>n.
  \end{equation}
  Under this assumption, $I$ must be infinite; fix a sequence $\{i_0, i_1, \ldots\}$.
  Set $E_1 = \{i_0\}$ and choose the corresponding $i_1 \in I \setminus E_1$ by \eqref{eq:fullness-bound-negation}.
  For $n \geq 1$, define $E_n=\{i_0,\ldots,i_{n-1}\}$ inductively.
  One obtains pairwise distinct elements $i_n\in I$ and compositions of arrows
  \begin{equation}\label{eq:fullness-unbounded-compositions}
    \mathfrak P_{i_n}:X_{i_n}\longrightarrow Y_{i_n},
    \qquad
    \ell(\mathfrak P_{i_n})>n
    \qquad(n\geq 1).
  \end{equation}
  For $i\notin\{i_n\mid n\in\mathbb N\}$, let $\mathfrak P_i$ be any identity arrow.
  This gives $([\mathfrak P_i])_{i\in I}\in \prod_{i\in I} \operatorname{Hom}_{\mathcal A_i[S_i^{-1}]} (Q_iX_i,Q_iY_i)$.
  Since $\Theta$ is full, there is a composition of arrows $\mathfrak R$ in $\langle\mathsf{Mor}(\prod_{i\in I}\mathcal A_i)\sqcup\overline{\prod_{i\in I}S_i}\rangle$ for which the following diagram holds:
  \[
    \xymatrix@C=20mm@R=2mm{
    \operatorname{Hom}_{
    (\prod_i\mathcal A_i)[(\prod_iS_i)^{-1}]}
    ((X_i)_i,(Y_i)_i)
    \ar[r]^-{\Theta}
    &
    \prod_i\operatorname{Hom}_{\mathcal A_i[S_i^{-1}]}
    (Q_iX_i,Q_iY_i)
    \\
    [\mathfrak R]\ar@{|->}[r]
    &
    ([\mathfrak R_i])_{i\in I}
    =([\mathfrak P_i])_{i\in I}.
    }
  \]
  Here $\mathfrak R_i$ is the $i$th coordinate of $\mathfrak R$.
  Let $m=|\mathfrak R|$.
  Then $\ell(\mathfrak P_i) \leq |\mathfrak R_i| \leq |\mathfrak R| = m$ holds.
  Taking $n>m$ in \eqref{eq:fullness-unbounded-compositions} gives $n<\ell(\mathfrak P_{i_n})\leq m<n$, a contradiction.

  \medskip

  Conversely, suppose that there exist a finite subset $F\subseteq I$ and $N\in\mathbb N$ such that
  \begin{equation}\label{eq:fullness-assumed-bound}
    \forall i\notin F,\quad
    \forall \mathfrak P\in
    \langle\mathsf{Mor}(\mathcal A_i)\sqcup\overline{S_i}\rangle, \quad \text{one has}
    \quad
    \ell(\mathfrak P)\leq N.
  \end{equation}
  It remains to prove that $\Theta$ is full.
  Take any morphism $([\mathfrak P_i])_{i\in I}:(Q_iX_i)_{i\in I}\longrightarrow(Q_iY_i)_{i\in I}$ in $\prod_{i \in I}\mathcal{A}_i[S_i^{-1}]$.
  For $i\notin F$, use \eqref{eq:fullness-assumed-bound} to choose $\mathfrak R_i$ such that $[\mathfrak R_i]=[\mathfrak P_i]$ and $|\mathfrak R_i|=\ell(\mathfrak P_i)\leq N$.
  For $i\in F$, set $\mathfrak R_i=\mathfrak P_i$.
  Let
  \begin{equation}\label{eq:fullness-global-composition-bound}
    M=\max\bigl(\{N\}\cup
    \{|\mathfrak R_i|\mid i\in F\}\bigr).
  \end{equation}
  Then $|\mathfrak R_i|\leq M$ for every $i\in I$.
  By inserting identity arrows, write each $\mathfrak R_i$ as an equivalent composition of arrows
  \begin{equation}\label{eq:common-alternating-representatives}
    \mathfrak R_i'
    =f_{i,M}\ast\overline{s_{i,M}}\ast f_{i,M-1}\ast
    \cdots\ast f_{i,1}\ast\overline{s_{i,1}}\ast f_{i,0},\qquad
    [\mathfrak R_i']=[\mathfrak R_i]=[\mathfrak P_i],\qquad
    |\mathfrak R_i'| = 2M+1.
  \end{equation}
  where all $f_{i,j}\in\mathsf{Mor}(\mathcal A_i)$ and $s_{i,j}\in S_i$.
  For $0\leq j\leq M$ and $1\leq j\leq M$, respectively, put
  \[
    f_j=(f_{i,j})_{i\in I}\in
    \mathsf{Mor}\left(\prod_{i\in I}\mathcal A_i\right),
    \qquad
    s_j=(s_{i,j})_{i\in I}\in\prod_{i\in I}S_i.
  \]
  The compositions of arrows in \eqref{eq:common-alternating-representatives} assemble coordinatewise into the composition of arrows $\mathfrak R=f_M\ast\overline{s_M}\ast f_{M-1}\ast\cdots\ast f_1\ast\overline{s_1}\ast f_0$.
  Thus
  \[
    \Theta ([\mathfrak R]) = ([\mathfrak R_i'])_{i\in I} =([\mathfrak P_i])_{i\in I}.
  \]
  Therefore $\Theta$ is surjective on every Hom set; that is, $\Theta$ is full.
\end{proof}

The following example shows that $\Theta$ need not be full.

\begin{example}\label{ex:product-localization-not-full}
  Let $I=\mathbb N_{>0}$, and for each $i\in I$ take the same one-object category $\mathcal C$, whose object is $\ast$ and whose morphisms form the monoid $(\mathbb N,+)$.
  Let $S=\{0,1\}\subseteq\mathbb N$, where $0 = \operatorname{Id}_\ast$ is the identity morphism.
  The following canonical functor is not full:
  \[
    \Theta:(\mathcal C^{\mathbb N_{>0}})[(S^{\mathbb N_{>0}})^{-1}]
    \longrightarrow(\mathcal C[S^{-1}])^{\mathbb N_{>0}}.
  \]
\end{example}

\begin{proof}
  One has $\operatorname{End}_{\mathcal C[S^{-1}]}(\ast)\cong(\mathbb Z,+)$.
  Indeed, take any $[\mathfrak P] \in \mathcal{C}[S^{-1}]$.
  Up to equivalence, one may assume that the only inverse arrow in $\mathfrak P$ is that of $1$, denoted by $-1$.
  It is clear that $n \ast 1$ is equivalent to $1 \ast n$, and hence $\overline 1 \ast n$ is equivalent to $n \ast \overline 1$.
  Also, $\overline 1 \ast 1$ is equivalent to $0$.
  Therefore $\mathfrak P$ is equivalent either to an arrow in $\mathsf{Mor}(\mathcal C)$ or to a composition of arrows of the form $\underbracket{\overline 1\ast\cdots\ast\overline 1}_{q\text{ copies}}$.
  It is straightforward to verify that
  \[
    \mathbb N \sqcup \big\{\underbracket{\overline 1\ast\cdots\ast\overline 1}_{q\text{ copies}}\mid q \geq 1\big\} \longrightarrow \mathbb Z, \qquad n \longmapsto n, \quad \underbracket{\overline 1\ast\cdots\ast\overline 1}_{q\text{ copies}} \longmapsto (-q)
  \]
  is a bijection of sets, under which $\ast$ corresponds to addition.
  Thus $\mathrm{End}_{\mathcal{C}[S^{-1}]}(\ast) \cong \mathbb Z$.

  \medskip

  For $n\geq1$, let $[\mathfrak P_n]=-n\in\mathbb Z$.
  There is no $\mathfrak R$ such that $\Theta ([\mathfrak R])=([\mathfrak P_n])_{n \geq 1}$.
  Otherwise, let $m$ be the number of inverse arrows in $\mathfrak R$.
  Then every component of $\mathfrak R$ contains at most $m$ inverse arrows; but $[\mathfrak R]$ has value $-(m+1)$ in its $(m+1)$st component, a contradiction.
  Hence $\Theta$ is not full.
\end{proof}

\subsection{A necessary and sufficient condition for \texorpdfstring{$\Theta$}{Theta} to be faithful}

For every $i\in I$, fix the following five types of elementary rewriting.
Here $f:X\longrightarrow Y$ and $g:Y\longrightarrow Z$ are arrows in $\mathsf{Mor}(\mathcal A_i)$, $a:M\longrightarrow N$ is an arrow in $\mathsf{Mor}(\mathcal A_i)\sqcup\overline{S_i}$, $s:X\longrightarrow Y$ is an arrow in $S_i$, and $\mathfrak U,\mathfrak V$ are compositions of arrows for which the following expressions are defined.
Either $\mathfrak U$ or $\mathfrak V$ may be omitted.
\begin{equation}\label{eq:gz-arrow-composition-relations}
  \begin{aligned}
    \mathfrak U\ast(g\ast f)\ast\mathfrak V
     & \longleftrightarrow
    \mathfrak U\ast(g\circ f)\ast\mathfrak V,          \\
    \mathfrak U\ast(\operatorname{Id}_N\ast a)\ast\mathfrak V
     & \longleftrightarrow
    \mathfrak U\ast a\ast\mathfrak V,                  \\
    \mathfrak U\ast(a\ast\operatorname{Id}_M)\ast\mathfrak V
     & \longleftrightarrow
    \mathfrak U\ast a\ast\mathfrak V,                  \\
    \mathfrak U\ast(s\ast\overline s)\ast\mathfrak V
     & \longleftrightarrow
    \mathfrak U\ast\operatorname{Id}_Y\ast\mathfrak V, \\
    \mathfrak U\ast(\overline s\ast s)\ast\mathfrak V
     & \longleftrightarrow
    \mathfrak U\ast\operatorname{Id}_X\ast\mathfrak V.
  \end{aligned}
\end{equation}
For any two equivalent compositions of arrows $\mathfrak P$ and $\mathfrak Q$, define
\[
  \delta(\mathfrak P,\mathfrak Q)
  =
  \min\left\{
  m\in\mathbb N\ \middle|\
  \begin{array}{c}
    \mathfrak P=\mathfrak R_0
    \longleftrightarrow\mathfrak R_1
    \longleftrightarrow\cdots
    \longleftrightarrow\mathfrak R_m=\mathfrak Q, \\
    \text{each step is one elementary rewriting from \eqref{eq:gz-arrow-composition-relations}}
  \end{array}
  \right\}.
\]
For $L\in\mathbb N$, define
\begin{equation}\label{eq:def-di-L}
  d_i(L) = \sup\left\{ \delta(\mathfrak P,\mathfrak Q) \mid [\mathfrak P]=[\mathfrak Q],\quad |\mathfrak P|\leq L,\quad|\mathfrak Q|\leq L \right\} \in\mathbb N\cup\{\infty\}.
\end{equation}
Here the supremum of the empty set is taken to be $0$.
For $L\geq1$, the set in \eqref{eq:def-di-L} is nonempty because $\delta(\operatorname{Id}_X,\operatorname{Id}_X)=0$.

If $\mathfrak{P}$ and $\mathfrak Q$ are related by at most $n$ elementary rewriting steps, write $\mathfrak{P}\xlongleftrightarrow{\leq n} \mathfrak{Q}$.

\begin{theorem}\label{thm:product-canonical-faithfulness}
  The canonical functor
  \[
    \Theta:\widetilde{\prod_{i\in I}Q_i}:
    \bigg(\prod_{i\in I}\mathcal A_i\bigg)
    \bigg[\bigg(\prod_{i\in I}S_i\bigg)^{-1}\bigg]
    \longrightarrow
    \prod_{i\in I}\mathcal A_i[S_i^{-1}]
  \]
  is faithful if and only if, for every $L\in\mathbb N$, there exist a finite subset $F_L\subseteq I$ and $D_L\in\mathbb N$ such that $d_i(L)\leq D_L$ for every $i\notin F_L$.
\end{theorem}

\begin{proof}
  Suppose first that $\Theta$ is faithful.
  If the stated boundedness condition fails, then there is an $L\in\mathbb N$ such that
  \begin{equation}\label{eq:faithfulness-bound-negation}
    \forall E\subseteq I\text{ finite},\ \
    \forall n\in\mathbb N,\ \
    \exists i\in I\setminus E,\ \
    \exists\mathfrak P_i,\mathfrak Q_i,\ \
    |\mathfrak P_i|,|\mathfrak Q_i|\leq L,\ \
    [\mathfrak P_i]=[\mathfrak Q_i],\ \
    \delta(\mathfrak P_i,\mathfrak Q_i)>n.
  \end{equation}
  One may assume that $L\geq1$.
  Under this assumption, $I$ must be infinite; fix a sequence $\{i_0, i_1, \ldots\}$.
  Set $E_1 = \{i_0\}$ and choose the corresponding $i_1 \in I \setminus E_1$ by \eqref{eq:faithfulness-bound-negation}.
  For $n \geq 1$, define $E_n=\{i_0,\ldots,i_{n-1}\}$ inductively.
  One obtains pairwise distinct elements $i_n\in I$ and corresponding pairs of compositions of arrows $(\mathfrak P_{i_n},\mathfrak Q_{i_n})$.

  For a composition of arrows $\mathfrak R$, let $\operatorname{sgn}(\mathfrak R)$ be the sign string obtained by recording its arrows in their composition order: an arrow in $\mathsf{Mor}(\mathcal A_i)$ is denoted by $+$ and an arrow in $\overline{S_i}$ by $-$.
  Since there are only finitely many sign strings of length at most $L$, one may pass to a subsequence of $\{i_k\}_{k \geq 0}$ and reindex it so that fixed sign strings $\sigma$ and $\tau$ satisfy
  \begin{equation}\label{eq:fixed-sign-pairs}
    \operatorname{sgn}(\mathfrak P_{i_n})=\sigma,\quad
    \operatorname{sgn}(\mathfrak Q_{i_n})=\tau,\quad
    [\mathfrak P_{i_n}]=[\mathfrak Q_{i_n}],\quad
    \delta(\mathfrak P_{i_n},\mathfrak Q_{i_n})>n
    \qquad(n\geq1),
  \end{equation}
  For $i\notin\{i_n\mid n\geq1\}$, choose an object $X_i\in\mathcal A_i$, and let $\mathfrak P_i$ and $\mathfrak Q_i$ be the compositions of arrows built from $\operatorname{Id}_{X_i}$ and $\overline{\operatorname{Id}_{X_i}}$ whose sign strings are $\sigma$ and $\tau$, respectively.
  Let $\mathfrak P = (\mathfrak P_i)_{i\in I}$ and $\mathfrak Q = (\mathfrak Q_i)_{i\in I}$.
  By construction,
  \begin{equation}\label{eq:same-theta-image}
    \Theta[\mathfrak P]
    =([\mathfrak P_i])_{i\in I}
    =([\mathfrak Q_i])_{i\in I}
    =\Theta[\mathfrak Q].
  \end{equation}
  By assumption, $\Theta$ is faithful.
  Therefore $[\mathfrak P]=[\mathfrak Q]$.
  Hence there is a chain of elementary rewritings $\mathfrak P=\mathfrak R_0 \longleftrightarrow\cdots \longleftrightarrow\mathfrak R_m=\mathfrak Q$.
  Projecting this chain to the $i_n$th coordinate and using \eqref{eq:fixed-sign-pairs}, one obtains $n<\delta(\mathfrak P_{i_n},\mathfrak Q_{i_n})\leq m$ for every $n \geq 1$.
  Taking $n=m+1$ gives a contradiction.
  This completes the proof in this direction.

  \medskip

  Conversely, assume that the stated boundedness condition holds.
  Take morphisms $[\mathfrak P],[\mathfrak Q]$ with the same source and target such that $\Theta[\mathfrak P]=\Theta[\mathfrak Q]$.
  Let $\mathfrak P_i,\mathfrak Q_i$ be the $i$th components of $\mathfrak P, \mathfrak Q$, respectively.
  Put $L=\max\{|\mathfrak P|,|\mathfrak Q|\}$.
  Then $|\mathfrak P_i|,|\mathfrak Q_i|\leq L$ and $[\mathfrak P_i]=[\mathfrak Q_i]$ for every $i\in I$.
  Choose $F_L,D_L$ as in the hypothesis, and let
  \begin{equation}\label{eq:uniform-rewriting-bound}
    D=\max\bigl(\{D_L\}\cup
    \{\delta(\mathfrak P_i,\mathfrak Q_i)\mid i\in F_L\}\bigr).
  \end{equation}
  Thus, for each $i\in I$, one may choose a chain of $m_i\leq D$ elementary rewritings
  \begin{equation}\label{eq:coordinate-rewriting-chain}
    \mathfrak P_i=\mathfrak R_{i,0}
    \longleftrightarrow\cdots
    \longleftrightarrow\mathfrak R_{i,m_i}=\mathfrak Q_i.
  \end{equation}
  An elementary rewriting increases the number of arrows by at most $1$, so $|\mathfrak R_{i,j}|\leq L+D$.
  For the $j$th step, let $\nu_{i,j},\varepsilon_{i,j},p_{i,j}$ denote, respectively, the row number of the relation in \eqref{eq:gz-arrow-composition-relations}, the direction in which it is used, and the position where it is applied.
  Record
  \[
    \kappa_i=\left(
    m_i;\
    (\nu_{i,j},\varepsilon_{i,j},p_{i,j})_{1\leq j\leq m_i};\
    (\operatorname{sgn}(\mathfrak R_{i,j}))_{0\leq j\leq m_i}
    \right).
  \]
  Since $m_i\leq D$, there are only five types of elementary relation, and the positions and sign strings all have length at most $L+D$, the invariant $\kappa_i$ takes only finitely many values.
  Partition $I$ into finitely many blocks according to these values, $I=I_1\amalg\cdots\amalg I_r$, where $i,i'\in I_t$ if and only if $\kappa_i=\kappa_{i'}$.
  Put $\mathcal B_t=\prod_{i\in I_t}\mathcal A_i$ and $T_t=\prod_{i\in I_t}S_i$.
  Within each $I_t$, every coordinate uses a relation from the same row, in the same direction, and with the same sign string at every step.
  Therefore the coordinatewise chains in \eqref{eq:coordinate-rewriting-chain} assemble into a chain of elementary rewritings in $\mathcal B_t$:
  \[
    \mathfrak P^{(t)}=\mathfrak R^{(t)}_0
    \longleftrightarrow\cdots
    \longleftrightarrow\mathfrak R^{(t)}_{m_t}=\mathfrak Q^{(t)},
    \qquad m_t=m_i\quad(i\in I_t),
  \]
  where $\mathfrak P^{(t)},\mathfrak Q^{(t)}$ are the projections of $\mathfrak P,\mathfrak Q$ to $I_t$, respectively.
  Hence $[\mathfrak P^{(t)}]=[\mathfrak Q^{(t)}]$ as morphisms in $\mathcal B_t[T_t^{-1}]$.

  From $\prod_{i\in I}\mathcal A_i\cong\prod_{t=1}^r\mathcal B_t$ and Lemma~\ref{lem:localization-binary-product}, one obtains an isomorphism of categories
  \[
    \Phi:
    \left(\prod_{i\in I}\mathcal A_i\right)
    \left[\left(\prod_{i\in I}S_i\right)^{-1}\right]
    \xlongrightarrow{\sim}
    \prod_{t=1}^r\mathcal B_t[T_t^{-1}].
  \]
  Thus $\Phi[\mathfrak P] =\bigl([\mathfrak P^{(t)}]\bigr)_{t=1}^r =\bigl([\mathfrak Q^{(t)}]\bigr)_{t=1}^r =\Phi[\mathfrak Q]$, hence $[\mathfrak P]=[\mathfrak Q]$.
  Therefore $\Theta$ is faithful.
\end{proof}

\begin{proposition}\label{prop:faithfulness-under-fullness}
  Suppose that $\Theta$ is full.
  By Theorem~\ref{thm:product-canonical-fullness}, there exist a finite subset $F_0\subseteq I$ and an integer $N_0\geq1$ such that
  \[
    \ell(\mathfrak P)\leq N_0
    \qquad
    \bigl(i\notin F_0,\
    \mathfrak P\in
    \langle\mathsf{Mor}(\mathcal A_i)\sqcup\overline{S_i}\rangle\bigr).
  \]
  Then $\Theta$ is faithful if and only if there exist a finite subset $F\subseteq I$ and $D\in\mathbb N$ such that $d_i(N_0+1)\leq D$ for every $i\notin F$.
\end{proposition}

\begin{proof}
  If $\Theta$ is faithful, take $L=N_0+1$ in Theorem~\ref{thm:product-canonical-faithfulness} and enlarge $F_L$ by $F_0$ to obtain the stated condition.

  \medskip

  Conversely, assume that the stated condition holds.
  Replacing $F$ by $F\cup F_0$, one may assume that $F_0\subseteq F$.
  Fix $i\notin F$ and take a composition of arrows $\mathfrak P=a_m\ast\cdots\ast a_1$ in $\langle\mathsf{Mor}(\mathcal A_i)\sqcup\overline{S_i}\rangle$.
  Set $\mathfrak R_1=a_1$, and use the length bound $N_0$ to choose successively representatives satisfying $[\mathfrak R_k]=[a_k\ast\mathfrak R_{k-1}]$ and $|\mathfrak R_k|\leq N_0$, for $2\leq k\leq m$.
  Since $|a_k\ast\mathfrak R_{k-1}|\leq N_0+1$, one has
  \[
    \delta(a_k\ast\mathfrak R_{k-1},\mathfrak R_k)
    \leq d_i(N_0+1)\leq D.
  \]
  If $m\geq2$, then for $k=2,\ldots,m$, replace $a_k\ast\mathfrak R_{k-1}$ by $\mathfrak R_k$ using at most $D$ elementary rewriting steps.
  This gives
  \[
    \mathfrak P
    =a_m\ast\cdots\ast a_2\ast\mathfrak R_1\xlongleftrightarrow{\leq D}
    a_m\ast\cdots\ast a_3\ast\mathfrak R_2\xlongleftrightarrow{\leq D}\cdots
    \xlongleftrightarrow{\leq D}
    \mathfrak R_m.
  \]
  There are $m-1$ replacements, each using at most $D$ elementary rewriting steps.
  If $m=1$, then $\mathfrak P=\mathfrak R_1$ and no rewriting is needed.
  Thus in all cases
  \[
    [\mathfrak R_m]=[\mathfrak P],
    \qquad |\mathfrak R_m|\leq N_0,
    \qquad
    \delta(\mathfrak P,\mathfrak R_m)\leq(m-1)D.
  \]
  If another composition of arrows $\mathfrak Q=b_n\ast\cdots\ast b_1$ in $\langle\mathsf{Mor}(\mathcal A_i)\sqcup\overline{S_i}\rangle$ satisfies $[\mathfrak P]=[\mathfrak Q]$, applying the same construction to $\mathfrak Q$ gives a composition of arrows $\mathfrak T_n$ satisfying
  \[
    [\mathfrak T_n]=[\mathfrak Q],
    \qquad |\mathfrak T_n|\leq N_0,
    \qquad
    \delta(\mathfrak Q,\mathfrak T_n)\leq(n-1)D.
  \]
  Since $[\mathfrak R_m]=[\mathfrak T_n]$ and $|\mathfrak R_m|,|\mathfrak T_n|\leq N_0$, one has $\delta(\mathfrak R_m,\mathfrak T_n)\leq d_i(N_0+1)\leq D$.
  Hence
  \[
    \delta(\mathfrak P,\mathfrak Q)
    \leq
    \delta(\mathfrak P,\mathfrak R_m)
    +\delta(\mathfrak R_m,\mathfrak T_n)
    +\delta(\mathfrak T_n,\mathfrak Q)\leq(m-1)D+D+(n-1)D.
  \]
  Thus $d_i(L)\leq(2L-1)D$ for every $L\geq1$ and every $i\notin F$.
  By Theorem~\ref{thm:product-canonical-faithfulness}, $\Theta$ is faithful.
\end{proof}

\begin{example}\label{ex:product-localization-not-faithful}
  Let $\mathcal C$ be the full subcategory of $\mathbb F_2$-vector spaces whose objects are $\{\mathbb F_2^d\}_{d \geq 0}$.
  For $d\geq0$, let
  \[
    j_d:\mathbb F_2^d\longrightarrow\mathbb F_2^{d+1},
    \qquad
    (x_1,\ldots,x_d) \longmapsto (x_1,\ldots,x_d,0),
  \]
  and take $S=\{j_d\mid d\geq0\} \cup \{\text{all isomorphisms}\}$.
  The following canonical functor is not faithful:
  \[
    \Theta:
    (\mathcal C^{\mathbb N_{>0}})
    [(S^{\mathbb N_{>0}})^{-1}]
    \longrightarrow
    (\mathcal C[S^{-1}])^{\mathbb N_{>0}}.
  \]
\end{example}

\begin{proof}
  Let $X=(\mathbb F_2^n)_{n\geq1}$, and take the morphisms $u=(\operatorname{Id}_{\mathbb F_2^n})_{n\geq1}$ and $v=(0_{n\times n})_{n\geq1}$.
  It will be shown that $\Theta[u]=\Theta[v]$ but $[u]\ne[v]$.

  \medskip

  On the one hand, for every $n\geq1$, the zero morphism $z_n:0\longrightarrow\mathbb F_2^n$ factors as
  \[
    0=\mathbb F_2^0
    \xrightarrow{j_0}\mathbb F_2
    \xrightarrow{j_1}\mathbb F_2^2
    \longrightarrow\cdots
    \xrightarrow{j_{n-1}}\mathbb F_2^n.
  \]
  Hence $[z_n]$ is invertible in $\mathcal C[S^{-1}]$.
  Cancelling $[z_n]$ from $\operatorname{Id}_{\mathbb F_2^n}\circ z_n =z_n =0_{n\times n}\circ z_n$ gives $[\operatorname{Id}_{\mathbb F_2^n}] =[0_{n\times n}]$.
  Therefore
  \[
    \Theta[u]
    =([\operatorname{Id}_{\mathbb F_2^n}])_{n\geq1}
    =([0_{n\times n}])_{n\geq1}
    =\Theta[v].
  \]

  \medskip

  On the other hand, define a relation on each Hom set of $\mathcal C^{\mathbb N_{>0}}$ by
  \[
    (f_n)_{n\geq1}\sim(g_n)_{n\geq1}
    \quad\Longleftrightarrow\quad
    \sup_{n\geq1}\operatorname{rank}(f_n-g_n)<\infty.
  \]
  By subadditivity of rank, this is an equivalence relation.
  For composable families of matrices, one has
  \[
    \operatorname{rank}(g_nf_n-g_n'f_n')
    =\operatorname{rank}\bigl( g_n(f_n-f_n')+(g_n-g_n')f_n'\bigr)\leq\operatorname{rank}(f_n-f_n')+\operatorname{rank}(g_n-g_n').
  \]
  Thus $\sim$ is compatible with composition and hence is a congruence on the category.
  Let $F:\mathcal C^{\mathbb N_{>0}} \longrightarrow \mathcal{C}^{\mathbb N_{>0}}/_\sim = \mathcal{D}$ be the corresponding quotient functor.
  For every $d\geq0$, take the projection $p_d=(I_d\ \ 0)$.
  Then $p_dj_d=I_d$, $j_dp_d= \begin{psmallmatrix}I_d&0\\0&0\end{psmallmatrix}$, and $\operatorname{rank}(j_dp_d-I_{d+1})=1$.
  Take any morphism $s=(s_n)_{n\geq1}:Y\longrightarrow Z$ in $S^{\mathbb N_{>0}}$.
  Coordinatewise, set
  \[
    r_n=
    \begin{cases}
      s_n^{-1}, & s_n\text{ is an isomorphism}, \\
      p_d,      & s_n=j_d.
    \end{cases}
  \]
  Then $r=(r_n)_{n\geq1}:Z\longrightarrow Y$ satisfies $rs=\operatorname{Id}_Y$, and
  \[
    \sup_{n\geq1}\operatorname{rank}
    (s_nr_n-\operatorname{Id}_{Z_n})\leq1.
  \]
  It follows that $F(r)F(s)=\operatorname{Id}_{FY}$ and $F(s)F(r)=\operatorname{Id}_{FZ}$.
  Therefore $F$ sends every morphism in $S^{\mathbb N_{>0}}$ to an isomorphism.
  By the universal property of localization, $F$ induces a functor $\widetilde F: (\mathcal C^{\mathbb N_{>0}}) [(S^{\mathbb N_{>0}})^{-1}] \longrightarrow \mathcal{D}$.
  But
  \[
    \operatorname{rank}(u_n-v_n)
    =\operatorname{rank}(I_n)=n,
    \qquad
    \sup_{n\geq1}\operatorname{rank}(u_n-v_n)=\infty,
  \]
  so $F(u)\ne F(v)$, and hence $[u]\ne[v]$.
  Together with $\Theta[u]=\Theta[v]$, $\Theta$ is not faithful.
\end{proof}

\begin{theorem}\label{cor:fraction-calculus-preservation}
  Let $\mathcal{A}$ be a category, assume that the class of morphisms $S \subseteq \mathsf{Mor}(\mathcal{A})$ contains all identity morphisms, and let $Q : \mathcal{A} \longrightarrow \mathcal{A}[S^{-1}]$ be the localization.
  Suppose that the category $\mathcal{A}$ and the class of morphisms $S$ satisfy the following two conditions:
  \begin{align*}
    N = \sup \big\{\ell(\mathfrak P) \mid \mathfrak P\in \langle\mathsf{Mor}(\mathcal A)\sqcup\overline{S}\rangle\big\} < \infty, \\
    \sup\big\{ \delta(\mathfrak P,\mathfrak Q) \mid [\mathfrak P]=[\mathfrak Q], \ \mathfrak P, \mathfrak Q\in \langle\mathsf{Mor}(\mathcal A)\sqcup\overline{S}\rangle, \ |\mathfrak P|, |\mathfrak Q| \leq N+1\big\} < \infty.
  \end{align*}
  Then the following statements hold:
  \begin{enumerate}[label=\textup{(\arabic*)}]
    \item If $\mathcal{A}$ has $I$-indexed coproducts and $S$ is closed under $I$-indexed coproducts, then $\mathcal{A}[S^{-1}]$ also has $I$-indexed coproducts and $Q$ preserves them;
    \item If $\mathcal{A}$ has $I$-indexed products and $S$ is closed under $I$-indexed products, then $\mathcal{A}[S^{-1}]$ also has $I$-indexed products and $Q$ preserves them.
  \end{enumerate}
\end{theorem}

\begin{proof}
  The proof is the same as that of Proposition~\ref{prop:finite-structure-localization}.
  It suffices to consider the coproduct case.
  If $\mathcal A$ is empty, the assertion is immediate, so assume that $\mathcal A$ is nonempty.
  Assume that $\mathcal{A}$ satisfies the two conditions above, that $\mathcal{A}$ has $I$-indexed coproducts, and that $S$ is closed under $I$-indexed coproducts.
  The coproduct functor $\coprod_I : \mathcal{A}^I \longrightarrow \mathcal{A}$ and the diagonal functor $\Delta^I : \mathcal{A} \longrightarrow \mathcal{A} ^I$ are respectively left and right adjoint.
  In particular, the right adjoint $\Delta^I$ sends $S$ to $S^I$; because $S$ is closed under $I$-indexed coproducts, the left adjoint $\coprod_I$ sends $S^I$ to $S$.
  By \cite[Lemma~1.1.6, p.~9]{KrauseHomologicalTheory}, or by the primary source \cite[Chapter~I, \S~1, Lemma~1.2, p.~7]{GabrielZisman}, the adjunction $\coprod_I \dashv \Delta^I$ therefore induces an adjunction on the localized categories
  \[
    \xymatrix@C=30mm@R=10mm{
    \mathcal A^I
    \ar@<0.8ex>[r]^{\coprod_I}
    \ar[d]_{Q^I}
    &
    \mathcal A
    \ar@<0.8ex>[l]^{\Delta^I}
    \ar[d]^{Q}
    \\
    (\mathcal A[S^{-1}])^I
    \ar@<0.8ex>[r]^-{\widehat {\coprod_I}}
    &
    \mathcal A[S^{-1}]
    \ar@<0.8ex>[l]^-{\widehat{\Delta^I}}
    } \qquad \xymatrix@R=10mm{
    Q \circ \coprod_I = \widehat{\coprod_I} \circ Q^I,\\
    Q^I \circ \Delta^I = \widehat{\Delta^I} \circ Q.
    }
  \]
  By the hypothesis, Theorem~\ref{thm:product-canonical-fullness}, and Proposition~\ref{prop:faithfulness-under-fullness}, the canonical functor $(\mathcal{A}^I)[(S^I)^{-1}] \longrightarrow (\mathcal{A}[S^{-1}])^I$ is an isomorphism of categories, so one may regard $Q^I$ as the localization functor of $\mathcal A^I$ at $S^I$.
  It is clear that $\widehat{\Delta^I}$ is still the diagonal functor, and $\widehat{\coprod_I}$ is its left adjoint.
  Therefore $\mathcal A[S^{-1}]$ has $I$-indexed coproducts, and $Q\circ\coprod_I=\widehat{\coprod_I}\circ Q^I$ shows that $Q$ preserves them.
\end{proof}

\section{Two applications for discrete indexing categories}\label{sec:discrete-applications}

This section gives two classes of sufficient conditions under which the canonical functor $\Theta_I$ is an equivalence of categories.
Throughout this section, $I$ is a set.

\subsection{calculus of fractions}

Gabriel and Zisman give the axioms for a calculus of fractions, fraction representations of localized morphisms, and a criterion for equality of two fractions in Chapter~I, \S\S~2.2--2.4 \cite[Chapter~I, \S\S~2.2--2.4, pp.~12--14]{GabrielZisman}.

\begin{definition}
  Let $\mathcal{A}$ be a category and $S \subseteq \mathsf{Mor}(\mathcal{A})$ a class of morphisms.
  The class $S$ is said to admit a right calculus of fractions if it contains all identity morphisms, is closed under composition, and satisfies the following conditions:
  \begin{enumerate}[label=\textup{(\arabic*)}]
    \item for $f:X\longrightarrow Y$ and $s:Y'\longrightarrow Y$, if $s\in S$, then there is a commutative square of the following form, with $t\in S$:
          \[
            \xymatrix@C=12mm@R=6mm{
            X'\ar[r]^{f'}\ar[d]_t&Y'\ar[d]^s\\
            X\ar[r]_f&Y.
            }
          \]
    \item for morphisms $f,g:X\longrightarrow Y$, if there is an $s\in S$ such that $s\circ f=s\circ g$, then there is a $t\in S$ such that $f\circ t=g\circ t$.
  \end{enumerate}
  A left calculus of fractions is obtained by dualizing these conditions.
\end{definition}

\begin{corollary}\label{thm:fraction-calculus-application}
  Let $\mathcal A$ be a category, $S\subseteq\mathsf{Mor}(\mathcal A)$ a class of morphisms, and $I$ a set.
  Let
  \[
    Q:\mathcal A\longrightarrow\mathcal A[S^{-1}],
    \qquad
    Q_I:\mathcal A^I\longrightarrow
    (\mathcal A^I)[(S^I)^{-1}].
  \]
  be the localizations.
  If $S$ admits a right calculus of fractions in $\mathcal{A}$, respectively a left calculus of fractions, then $S^I$ admits a right calculus of fractions in $\mathcal A^I$, respectively a left calculus of fractions, and the canonical functor obtained by factoring $Q^I$ through $Q_I$ is an isomorphism:
  \[
    \Theta_I:
    (\mathcal A^I)[(S^I)^{-1}]
    \longrightarrow(\mathcal A[S^{-1}])^I,
    \qquad
    \Theta_I\circ Q_I=Q^I.
  \]

  Assume that $S$ admits a left or right calculus of fractions.
  Furthermore, if $\mathcal A$ has $I$-indexed (co)products and $S$ is closed under $I$-indexed (co)products, then $\mathcal A[S^{-1}]$ has $I$-indexed (co)products and $Q$ preserves them.
\end{corollary}

\begin{proof}
  It suffices to prove the case of a right calculus of fractions; the left calculus case is dual.
  The axioms for a calculus of fractions can be verified coordinatewise, so $S^I$ clearly admits a right calculus of fractions.

  \medskip
  By the right calculus of fractions, every $\mathfrak P\in\langle\mathsf{Mor}(\mathcal A)\sqcup\overline S\rangle$ is equivalent to $f\ast\overline s$, where $s\in S$ \cite[Chapter~I, \S~2.4, p.~14]{GabrielZisman}.
  Taking $F=\varnothing$ and $N=2$ in Theorem~\ref{thm:product-canonical-fullness}, $\Theta_I$ is full.

  \medskip

  By Proposition~\ref{prop:faithfulness-under-fullness}, it remains only to verify that $d_i(3)$ has a uniform upper bound in order to show that $\Theta_I$ is faithful.
  First consider the rewritings needed to refine a fraction.
  If $s:Z\longrightarrow X$, $f:Z\longrightarrow Y$, $a:T\longrightarrow Z$, and $s,s\circ a\in S$, then
  \[
    f\ast\overline s
    \xlongleftrightarrow{\leq2}
    f\ast\overline s\ast(s\circ a)\ast\overline{s\circ a}\xlongleftrightarrow{\leq4}
    (f\circ a)\ast\overline{s\circ a}.
  \]
  Thus every such refinement uses at most $6$ elementary rewriting steps.
  If $[f\ast\overline s]=[g\ast\overline t]$, then the equality criterion for right fractions \cite[Chapter~I, \S\S~2.3--2.4, pp.~13--14]{GabrielZisman} gives $a,b$ such that
  \[
    s\circ a=t\circ b\in S,
    \qquad f\circ a=g\circ b.
  \]
  Each of the two right fractions reaches this common refinement in at most $6$ steps, so the number of elementary rewriting steps needed to pass to a common denominator is bounded by $\delta(f\ast\overline s,g\ast\overline t)\leq12$.

  A single arrow can be rewritten as a right fraction in at most $3$ elementary steps.
  Adding a forward arrow to the left of $f\ast\overline s$ requires only one composition; when adding $\overline t$, the axioms for right fractions give $a\in S$ and $h$ such that $f\circ a=t\circ h$, and hence
  \[
    \overline t\ast f\ast\overline s
    \xlongleftrightarrow{\leq6}
    \overline t\ast(f\circ a)\ast\overline{s\circ a}=\overline t\ast(t\circ h)\ast\overline{s\circ a}
    \xlongleftrightarrow{\leq3}
    h\ast\overline{s\circ a}.
  \]
  Therefore, a composition of at most $3$ arrows can be rewritten as a right fraction in at most $3+2\cdot9=21$ steps.
  For arbitrary $\mathfrak P,\mathfrak Q\in \langle\mathsf{Mor}(\mathcal A)\sqcup\overline S\rangle$, one has
  \[
    \begin{gathered}
      [\mathfrak P]=[\mathfrak Q],\quad
      |\mathfrak P|,|\mathfrak Q|\leq3
      \quad\Longrightarrow\quad
      \delta(\mathfrak P,\mathfrak Q)\leq21+12+21=54.
    \end{gathered}
  \]
  Thus $d_i(3)\leq54$ for all $i\in I$.
  Taking $N_0=2$, $F_0=F=\varnothing$, and $D=54$ in Proposition~\ref{prop:faithfulness-under-fullness}, one concludes that $\Theta_I$ is faithful.
  Since $\Theta_I$ is bijective on objects, it is an isomorphism of categories.

  \medskip

  Finally, $\ell(\mathfrak P)\leq2$ and $d_i(3)\leq54$ verify the two conditions in Theorem~\ref{cor:fraction-calculus-preservation}.
  Hence, whenever the corresponding $I$-indexed (co)products exist in $\mathcal A$ and $S$ is closed under them, they exist in $\mathcal A[S^{-1}]$ and are preserved by $Q$.
\end{proof}

\subsection{Model structures}

The term model structure below is used in the sense of Definition~\ref{def:model-structure}.
A model structure on a category $\mathcal A$ is determined by three classes of morphisms $(\mathsf{Cofib},\mathsf{Fib},\mathsf{Weq})$ containing all identity morphisms and closed under composition; their elements are called cofibrations, fibrations, and weak equivalences, respectively.
Morphisms in $\mathsf{TCofib} = \mathsf{Cofib}\cap\mathsf{Weq}$ are called trivial cofibrations, and morphisms in $\mathsf{TFib} = \mathsf{Fib}\cap\mathsf{Weq}$ are called trivial fibrations.
The following axioms for $(\mathsf{Cofib}, \mathsf{Fib}, \mathsf{Weq})$ come from Quillen's definition of a closed model category \cite[Chapter~I, \S~1]{Quillen}; see also \cite[Section~3, Definition~3.3]{DwyerSpalinski}.

\begin{definition}\label{def:model-structure}
  A triple $(\mathsf{Cofib},\mathsf{Fib},\mathsf{Weq})$ is called a model structure on a category $\mathcal A$ if the following conditions hold.
  \begin{enumerate}
    \item[CM1] For every composable pair of morphisms $g \circ f$, if two of $(f,g,g \circ f)$ are weak equivalences, then so is the third.
    \item[CM2] The three classes of morphisms $(\mathsf{Cofib},\mathsf{Fib},\mathsf{Weq})$ are closed under retracts.
          A morphism $f' : X' \longrightarrow Y'$ is called a retract of $f : X \longrightarrow Y$ in the arrow category if there exist morphisms represented by the dotted arrows that make the following diagram commute:
          \[
            \xymatrix@R=15pt{
            X' \ar@{.>}[r] \ar@{->}[d]_-{f'} \ar@/^/@{=}[rr]^{\operatorname{Id}_{X'}} & X \ar@{.>}[r] \ar@{->}[d]_-{f} & X' \ar@{->}[d]^-{f'} \\
            Y' \ar@{.>}[r] \ar@/_/@{=}[rr]_{\operatorname{Id}_{Y'}} & Y \ar@{.>}[r] & Y'
            }
          \]
    \item[CM3] For every commutative square
          \[
            \xymatrix@C=14mm@R=7mm{
            X\ar[r]^a\ar[d]_i&E\ar[d]^p\\
            Y\ar[r]_b\ar@{..>}[ur]^h&B
            },
          \]
          if either $i\in\mathsf{Cofib}$ and $p \in \mathsf{TFib}$, or $i \in \mathsf{TCofib}$ and $p \in \mathsf{Fib}$, there exists $h$ such that $h\circ i=a$ and $p\circ h=b$.
    \item[CM4] Every morphism $f : X \longrightarrow Y$ has factorizations of the following two forms:
          \begin{enumerate}
            \item[(i)] there exist $i : X \longrightarrow Z \in \mathsf{Cofib}$ and $p : Z \longrightarrow Y \in \mathsf{TFib}$ such that $f = p \circ i$;
            \item[(ii)] there exist $i' : X \longrightarrow Z' \in \mathsf{TCofib}$ and $p' : Z' \longrightarrow Y \in \mathsf{Fib}$ such that $f = p' \circ i'$.
          \end{enumerate}
  \end{enumerate}
\end{definition}

Quillen's original definition also requires the category to have all finite limits and finite colimits \cite[Chapter~I, \S~1]{Quillen}.
Egger proved that the existence of finite products and finite coproducts is sufficient for constructing the homotopy category \cite[p.~1 and Theorem~3.2]{EggerModelCategories}.
In what follows, $\mathcal A$ is assumed to have finite products and finite coproducts.
In particular, $\mathcal A$ has an initial object $\bot$ and a terminal object $\top$.

\begin{definition}
  Let the full subcategories of cofibrant objects, fibrant objects, trivially cofibrant objects, and trivially fibrant objects be denoted, respectively, by
  \[
    \begin{aligned}
      \mathcal C
       & =\{X\in\mathcal A\mid \bot\longrightarrow X\in\mathsf{Cofib}\},  &
      \mathcal F
       & =\{X\in\mathcal A\mid X\longrightarrow\top\in\mathsf{Fib}\},       \\
      \mathsf T\mathcal C
       & =\{X\in\mathcal A\mid \bot\longrightarrow X\in\mathsf{TCofib}\}, &
      \mathsf T\mathcal F
       & =\{X\in\mathcal A\mid X\longrightarrow\top\in\mathsf{TFib}\}.
    \end{aligned}
  \]
  For the definitions of cofibrant and fibrant objects, see \cite[Section~3, Remark~3.4]{DwyerSpalinski}; trivially cofibrant and trivially fibrant objects are defined in the same way.
\end{definition}

\begin{lemma}\label{lem:model-replacements}
  Let $\mathcal A$ carry a model structure.
  For every object $X\in\mathcal A$, there exist a cofibrant object $QX$ and a trivial fibration $p_X : QX \longrightarrow X$, as well as a fibrant object $RX$ and a trivial cofibration $i_X : X \longrightarrow RX$.
\end{lemma}

\begin{proof}
  Apply the factorization axioms in (CM4) to the morphisms $\bot\longrightarrow X$ and $X \longrightarrow \top$.
\end{proof}

For every object $X$, fix such factorizations once and for all, and consistently write
\[
  p_X:QX\longrightarrow X,\qquad i_X:X\longrightarrow RX
\]
for a cofibrant replacement and a fibrant replacement, respectively.
This notation follows \cite[Section~5, Lemmas~5.1 and~5.3, pp.~24--25]{DwyerSpalinski}.
Since the factorization axiom does not require the factorizations to be functorial, $Q$ and $R$ here are initially only choices on objects; for a given morphism, the required lift is chosen separately.

\medskip

Quillen defined the homotopy category as the Gabriel--Zisman localization at the weak equivalences and proved that it is equivalent to the category of homotopy classes between cofibrant-fibrant objects \cite[Chapter~I, \S~1, Theorem~1, pp.~24--27]{Quillen}.
Under the hypotheses of finite products and finite coproducts used in this paper, Egger later proved that the same conclusion remains valid \cite[Theorem~3.2]{EggerModelCategories}.

\begin{theorem}\label{thm:quillen-homotopy-category}
  Suppose that $\mathcal A$ has finite products and finite coproducts and carries a model structure $(\mathsf{Cofib},\mathsf{Fib},\mathsf{Weq})$.
  Let $\gamma:\mathcal A\longrightarrow \operatorname{Ho}(\mathcal A) =\mathcal A[\mathsf{Weq}^{-1}]$ be the localization.
  Then, for every $X \in \mathcal{C}$ and $Y \in \mathcal{F}$, every morphism in $\mathrm{Hom}_{\mathsf{Ho}(\mathcal{A})}(X,Y)$ has the form $\gamma (f)$ for some $f \in \mathrm{Hom}_\mathcal{A}(X,Y)$.
  Moreover, two morphisms $f,g:X\longrightarrow Y$ are equal in $\operatorname{Ho}(\mathcal A)$ if and only if they are homotopic.
\end{theorem}

\begin{proof}
  This is Quillen's fundamental theorem on model structures \cite[Chapter~I, \S~1, Theorem~1 and Corollary~1, pp.~24--27]{Quillen}; here the weakened version that requires only the existence of finite products and finite coproducts is used \cite[Theorem~3.2]{EggerModelCategories}.
  Egger's comparison is induced by the inclusion of the cofibrant-fibrant objects followed by the localization functor.
  Its full faithfulness therefore identifies the Hom set in the localization with the set of homotopy classes of morphisms from $X$ to $Y$, which gives both assertions.
\end{proof}

\begin{lemma}\label{lem:product-model-structure}
  Let $I$ be an indexing set, and let $\bigl(\mathcal A_i,(\mathsf{Cofib}_i,\mathsf{Fib}_i,\mathsf{Weq}_i)\bigr)_{i\in I}$ be a family of model categories with finite products and finite coproducts.
  Then the product category $\prod_{i\in I}\mathcal A_i$ has the coordinatewise model structure
  \[
    \left(\prod_{i\in I}\mathsf{Cofib}_i,\,
    \prod_{i\in I}\mathsf{Fib}_i,\,
    \prod_{i\in I}\mathsf{Weq}_i\right).
  \]
  This model structure is called the product model structure \cite[Example~1.1.6, p.~4]{HoveyModelCategories}.
\end{lemma}

For ordinary model categories, the following comparison equivalence for discrete indexing categories is recorded in \cite[Section~2.5, equation~(2.7) and footnote~14, pp.~41--42]{RiehlHomotopicalCategories}.
The proof below also applies under the weakened finite-(co)product hypotheses of Definition~\ref{def:model-structure}.

\begin{theorem}\label{thm:model-category-application}
  Let $\mathcal A$ be a model category and $I$ a discrete indexing set.
  Let
  \[
    \gamma:\mathcal A\longrightarrow\operatorname{Ho}(\mathcal A),
    \qquad
    \gamma_I:\mathcal A^I\longrightarrow
    \operatorname{Ho}(\mathcal A^I)
    =(\mathcal A^I)[(\mathsf{Weq}^I)^{-1}].
  \]
  be the localizations.
  The canonical functor obtained by factoring $\gamma^I$ through $\gamma_I$,
  \[
    \Theta_I:
    \operatorname{Ho}(\mathcal A^I)
    \longrightarrow(\operatorname{Ho}(\mathcal A))^I,
    \qquad
    \Theta_I\circ\gamma_I=\gamma^I
  \]
  is an isomorphism of categories.
  Equivalently, for any two families of objects $X=(X_i)_{i\in I}$ and $Y=(Y_i)_{i\in I}$, there is a canonical bijection
  \[
    \operatorname{Hom}_{\operatorname{Ho}(\mathcal A^I)}
    (\gamma_IX,\gamma_IY)
    \xlongrightarrow{\sim}
    \prod_{i\in I}
    \operatorname{Hom}_{\operatorname{Ho}(\mathcal A)}
    (\gamma X_i,\gamma Y_i).
  \]
\end{theorem}

\begin{proof}
  Take a composition of arrows $\mathfrak P$ from $X$ to $Y$ in $\langle\mathsf{Mor}(\mathcal A)\sqcup\overline{\mathsf{Weq}}\rangle$.
  By the cofibrant replacement, the fibrant replacement, and Theorem~\ref{thm:quillen-homotopy-category}, there is an $h:QX\longrightarrow RY$ such that $\gamma(h)=\gamma(i_Y)\circ[\mathfrak P]\circ\gamma(p_X)$.
  Therefore $\mathfrak P$ is equivalent to the following composition of three arrows:
  \[
    \xymatrix@C=14mm{
    X&
    QX\ar[l]_{p_X}\ar[r]^h&
    RY&
    Y\ar[l]_{i_Y}
    },
    \qquad
  [\mathfrak P]=[\overline{i_Y}\ast h\ast\overline{p_X}],
    \qquad
    \ell(\mathfrak P)\leq3.
  \]
  Taking $F=\varnothing$ and $N=3$ in Theorem~\ref{thm:product-canonical-fullness}, one concludes that $\Theta_I$ is full.

  \medskip

  By Proposition~\ref{prop:faithfulness-under-fullness}, it remains only to verify that $d_i(4)$ has a uniform upper bound.
  The notation $\xlongleftrightarrow{\leq m}$ continues to denote a sequence of at most $m$ elementary rewriting steps from \eqref{eq:gz-arrow-composition-relations}.
  First let $C \in \mathcal{C}$ be cofibrant, let $T \in \mathcal{F}$ be fibrant, and let $h_0,h_1:C\longrightarrow T$ satisfy $\gamma(h_0)=\gamma(h_1)$.
  By Theorem~\ref{thm:quillen-homotopy-category} and \cite[Lemma~4.6]{DwyerSpalinski}, there is a commutative diagram
  \[
    \xymatrix@C=12mm@R=7mm{
    C\amalg C\ar[r]^{(h_0, h_1)}\ar[d]_{(\operatorname{Id}_C, \operatorname{Id}_C)} \ar@{..>}[dr]^{(i_0, i_1)} &T\\
    C&\operatorname{Cyl}(C)\ar[u]_H\ar[l]_p
    }
  \]
  where $(i_0,i_1)\in\mathsf{Cofib}$ and $p\in \mathsf{TFib}$.
  Hence
  \[
    i_\varepsilon
    \xlongleftrightarrow{\leq2}
    \overline p\ast p\ast i_\varepsilon
    \xlongleftrightarrow{\leq2}\overline p,
    \qquad
    h_\varepsilon
    \longleftrightarrow H\ast i_\varepsilon
    \xlongleftrightarrow{\leq4}H\ast\overline p.
  \]
  Therefore
  \begin{equation}\label{eq:model-homotopy-rewriting}
    \delta(h_0,h_1)\leq10.
  \end{equation}

  \medskip

  Next take $C,D \in \mathcal{C}$, a weak equivalence $s:C\longrightarrow D$, an object $T \in \mathcal{F}$, and a morphism $h:C\longrightarrow T$.
  By (CM4) and (CM1), choose
  \[
    s=p\circ u,\qquad C\xlongrightarrow{u}Z\xlongrightarrow{p}D,\qquad u\in \mathsf{TCofib},\quad p\in \mathsf{TFib}.
  \]
  The following two lifting squares give $v:Z\longrightarrow T$ and $r:D\longrightarrow Z$:
  \[
    \xymatrix@C=12mm@R=7mm{
    C\ar[r]^h\ar[d]_u&T\ar[d]\\
    Z\ar[r]\ar@{..>}[ur]^v&\top
    }
    \qquad
    \xymatrix@C=12mm@R=7mm{
    \bot\ar[r]\ar[d]&Z\ar[d]^p\\
    D\ar[r]_{\operatorname{Id}_D}\ar@{..>}[ur]^r&D
    },
    \qquad
    v\circ u=h,\quad p\circ r=\operatorname{Id}_D.
  \]
  Therefore the following rewritings hold:
  \begin{equation}\label{eq:model-weak-inverse-rewriting}
    h\ast\overline s
    =(v\circ u)\ast\overline{p\circ u}
    \xlongleftrightarrow{\leq6}
    v\ast\overline p
    \xlongleftrightarrow{\leq5}v\circ r.
  \end{equation}

  \medskip

  For every morphism $f:U\longrightarrow V$, the lifting axiom gives a morphism $\widetilde f:QU\longrightarrow QV$:
  \[
    \xymatrix@C=16mm@R=7mm{
    \bot\ar[r]\ar[d]&QV\ar[d]^{p_V}\\
    QU\ar[r]_{f\circ p_U}
    \ar@{..>}[ur]^{\widetilde f}&V
    },
    \qquad
    p_V\circ\widetilde f=f\circ p_U.
  \]
  This commutative square gives the first rewriting below for every $f\in\mathsf{Mor}(\mathcal A)$.
  If in addition $f\in\mathsf{Weq}$, then $\widetilde f\in\mathsf{Weq}$, and let $\overline f:V\longrightarrow U$ and $\overline{\widetilde f}:QV\longrightarrow QU$ denote the inverse arrows in the localization corresponding to $f$ and $\widetilde f$, respectively; in this case one also has the second rewriting:
  \begin{equation}\label{eq:model-cofibrant-arrow-rewriting}
    f\xlongleftrightarrow{\leq4} p_V\ast\widetilde f\ast\overline{p_U}
    \quad\bigl(f\in\mathsf{Mor}(\mathcal A)\bigr), \qquad
    \overline f\xlongleftrightarrow{\leq8}
    p_U\ast\overline{\widetilde f}\ast\overline{p_V}
    \quad\bigl(f\in\mathsf{Weq}\bigr).
  \end{equation}

  Now take $\mathfrak P=a_m\ast\cdots\ast a_1 \in\langle\mathsf{Mor}(\mathcal A)\sqcup\overline{\mathsf{Weq}}\rangle$, where $1\leq m\leq4$, $a_k:X_{k-1}\longrightarrow X_k$, $X_0=X$, and $X_m=Y$.
  Let $p_k=p_{X_k}$.
  Applying \eqref{eq:model-cofibrant-arrow-rewriting} to each $a_k$ gives
  \[
    \begin{aligned}
      \mathfrak P
       & \xlongleftrightarrow{\leq8m}
      (p_m\ast b_m\ast\overline{p_{m-1}})
      \ast\cdots\ast(p_1\ast b_1\ast\overline{p_0})   \\
       & \xlongleftrightarrow{\leq2(m-1)}
      p_m\ast b_m\ast\cdots\ast b_1\ast\overline{p_0} \\
       & \xlongleftrightarrow{\leq3}
      \overline{i_Y}\ast(i_Y\circ p_m)
      \ast b_m\ast\cdots\ast b_1\ast\overline{p_0}.
    \end{aligned}
  \]
  where $b_k$ is an arrow from $QX_{k-1}$ to $QX_k$: it is either a morphism in $\mathcal A$ or the inverse arrow corresponding to a weak equivalence.
  Starting with $(i_Y\circ p_m)\ast b_m$, process the arrows $b_k$ successively from left to right: compose once for a forward arrow, and apply \eqref{eq:model-weak-inverse-rewriting} for an inverse arrow.
  In at most $11m$ steps, this gives a morphism $h:QX\longrightarrow RY$.
  Therefore
  \begin{equation}\label{eq:model-three-arrow-rewriting}
    \delta\bigl(\mathfrak P,
    \overline{i_Y}\ast h\ast\overline{p_X}\bigr)
    \leq8m+2(m-1)+3+11m\leq85.
  \end{equation}
  If $\mathfrak Q\in \langle\mathsf{Mor}(\mathcal A)\sqcup\overline{\mathsf{Weq}}\rangle$ satisfies $[\mathfrak P]=[\mathfrak Q]$ and $|\mathfrak Q|\leq4$, use the same cofibrant replacement $p_X$ and fibrant replacement $i_Y$ for $\mathfrak Q$ to obtain a morphism $h':QX\longrightarrow RY$.
  Then
  \[
    [h]=\gamma(i_Y)\circ[\mathfrak P]\circ\gamma(p_X)
    =\gamma(i_Y)\circ[\mathfrak Q]\circ\gamma(p_X)=[h'].
  \]
  Combining \eqref{eq:model-homotopy-rewriting} and \eqref{eq:model-three-arrow-rewriting}, one obtains
  \[
    \delta(\mathfrak P,\mathfrak Q)\leq85+10+85=180,
    \qquad d_i(4)\leq180\quad(i\in I).
  \]
  Taking $N_0=3$, $F_0=F=\varnothing$, and $D=180$ in Proposition~\ref{prop:faithfulness-under-fullness}, one concludes that $\Theta_I$ is faithful.
  Since $\Theta_I$ is bijective on objects, it is an isomorphism of categories.
\end{proof}

\begin{corollary}\label{cor:model-category-preservation}
  Let $(\mathcal A, (\mathsf{Cofib}, \mathsf{Fib}, \mathsf{Weq}))$ be a model category.
  Let $\gamma:\mathcal A\longrightarrow\operatorname{Ho}(\mathcal A)$ be the localization.
  Fix a set $I$.
  \begin{enumerate}[label=\textup{(\arabic*)}]
    \item If $\mathcal A$ has $I$-indexed coproducts and $\mathsf{Weq}$ is closed under these coproducts, then $\operatorname{Ho}(\mathcal A)$ has $I$-indexed coproducts and $\gamma$ preserves them;
    \item If $\mathcal A$ has $I$-indexed products and $\mathsf{Weq}$ is closed under these products, then $\operatorname{Ho}(\mathcal A)$ has $I$-indexed products and $\gamma$ preserves them.
  \end{enumerate}
  In particular, if $\mathcal A$ has all coproducts (products) and $\mathsf{Weq}$ is closed under them, then $\gamma$ preserves all coproducts (products).
\end{corollary}

\begin{proof}
  Take $S=\mathsf{Weq}$ and $Q=\gamma$ in Theorem~\ref{cor:fraction-calculus-preservation}.
  The preceding proof gives $\ell(\mathfrak P)\leq3$ and $d_i(4)\leq180$, so the two conditions in Theorem~\ref{cor:fraction-calculus-preservation} hold, and the stated conclusions about $I$-indexed (co)products and their preservation follow.
\end{proof}

\section{Conjectures and further questions}\label{sec:conjectures-and-problems}
The following conjecture was left open in the first arXiv version of this paper.
Recently, Ruixiao Geng informed the author that he had found an example confirming the conjecture.
The author thanks him for answering this question.
A variant of Geng's construction is given in Subsection~\ref{subsec:nonautomatic-additivity-solution}.

\begin{conjecture}\label{conj:nonautomatic-additivity}
  There exist an additive category $\mathcal A$ and a class of morphisms $S$ such that both $\mathcal A$ and $\mathcal A[S^{-1}]$ are additive categories, while the localization functor $Q:\mathcal A\longrightarrow\mathcal A[S^{-1}]$ is not an additive functor.
\end{conjecture}

There is in fact an example in which $\mathcal A$ is preadditive and the ordinary category $\mathcal A[S^{-1}]$ can be equipped with a preadditive structure for which $Q$ does not preserve addition of the Hom sets.

\begin{example}\label{ex:preadditive-nonadditive-localization}
  Let $\mathcal A$ be the one-object category associated with the algebra $\mathbb F_5\times\mathbb F_2$: the category $\mathcal A$ has a single object $*$ and $\operatorname{End}_{\mathcal A}(*)=\mathbb F_5\times\mathbb F_2$.
  Composition and addition of morphisms are given by multiplication and addition in this algebra, respectively.
  Let $e=(1,0)$ and take $S=\{e\}$.
  Then $\mathcal A[S^{-1}]$ can be equipped with a preadditive structure for which the localization functor $Q:\mathcal A\longrightarrow \mathcal{A}[S^{-1}]$ does not preserve addition.
\end{example}

\begin{proof}
  Let $\mathcal B$ be the one-object category associated with the algebra $\mathbb F_5$.
  Define a functor
  \[
    F:\mathcal A\longrightarrow\mathcal B,
    \qquad F(*)=*,\qquad F(a,b)=a^3.
  \]
  Since $(ac)^3=a^3c^3$, $F$ is a functor, and $F(e)=1$.
  The universal property of localization therefore gives a unique functor
  \[
    \widetilde F:\mathcal A[S^{-1}]\longrightarrow\mathcal B,
    \qquad
    \widetilde F\circ Q=F.
  \]
  By construction, $Q((1,0))$ is an isomorphism.
  It is straightforward to verify that the following assignment satisfies the identity and composition laws and therefore defines a functor:
  \[
    G:\mathcal B\longrightarrow\mathcal A[S^{-1}],
    \qquad
    G(*)=Q(*),\qquad G(a)=Q(a^3,0).
  \]
  Then $\widetilde F\bigl(G(a)\bigr) =F(a^3,0)=a^9=a$.
  Hence $\widetilde F\circ G=\operatorname{Id}_{\mathcal B}$.
  On the other hand, $Q(a,b)=Q(e)\circ Q(a,b)=Q(a,0)$, and therefore
  \[
    G\bigl(\widetilde F(Q(a,b))\bigr) =G(a^3)=Q(a^9,0)=Q(a,0)=Q(a,b).
  \]
  Thus $(G\circ\widetilde F)\circ Q=Q$.
  By uniqueness in the universal property of localization, $G\circ\widetilde F =\operatorname{Id}_{\mathcal A[S^{-1}]}$.
  Hence $\widetilde F$ is an isomorphism of categories.
  Transport the addition of $\mathcal B$ along this isomorphism: for morphisms $u,v$ in the unique Hom set of $\mathcal A[S^{-1}]$, define
  \[
    u+_{S}v
    =G\bigl(\widetilde F(u)+\widetilde F(v)\bigr).
  \]
  With this addition, $\mathcal A[S^{-1}]$ is preadditive and $\widetilde F$ is an additive isomorphism.

  \medskip

  It remains to prove that $Q$ is not additive.
  If $Q$ preserved addition, then $F=\widetilde F\circ Q$ would also preserve addition.
  But
  \[
    F(e+e)=F(2,0)=2^3=3
    \neq2=1+1=F(e)+F(e) \qquad \text{in $\mathbb F_5$}
  \]
  Thus $Q$ does not preserve addition.
\end{proof}

For a set $I$, the product category $\mathcal{A}^I$ is precisely the functor category $\mathrm{Funct}(I, \mathcal{A})$.
What happens for a general indexing category $J$?

\begin{problem}\label{prob:functor-localization}
Let $J$ be a category.
Let $S^J$ be the class of all natural transformations in $\mathcal A^J$ that belong pointwise to $S$.
When is the canonical functor induced by localization
\[
  \Theta_{J}:(\mathcal A^{J})[(S^{J})^{-1}]\longrightarrow(\mathcal A[S^{-1}])^{J}
\]
an equivalence of categories, and when is it essentially surjective, full, or faithful, respectively?
\end{problem}

This property already fails in general for arrow categories.

\begin{example}\label{ex:morphism-category-localization}
  Let $\{a,s\}^*$ be the free monoid generated by the letters $a$ and $s$, and let $\mathcal A$ be the corresponding one-object category.
  Let $S=\{s^n\mid n\geq0\}$, and let $[1]$ be the category generated by the arrow $0\longrightarrow1$, so that $\mathcal A^{[1]}$ is the arrow category of $\mathcal A$.
  Then the canonical functor induced by localization
  \[
    \Theta_{[1]}:
    (\mathcal A^{[1]})[(S^{[1]})^{-1}]
    \longrightarrow
    (\mathcal A[S^{-1}])^{[1]}
  \]
  is not essentially surjective and hence is not an equivalence of categories.
\end{example}

\begin{proof}
  The localization $\mathcal A[S^{-1}]$ still has a single object, and its endomorphism monoid is $M=\langle a,s,\overline s\mid s\overline s=1=\overline s s\rangle \cong\mathbb N*\mathbb Z$.

  Observe that $(\mathcal A^{[1]})[(S^{[1]})^{-1}]$ and $\mathcal A^{[1]}$ have the same objects, namely all morphisms of $\mathcal{A}$.
  Take any object $m\in\{a,s\}^*$ of $(\mathcal A^{[1]})[(S^{[1]})^{-1}]$; the functor $\Theta_{[1]}$ sends it to the object $[m]$.
  A functor that distinguishes $[a\overline s a]$ from $[m]$ will now be constructed.
  Take
  \[
    V=\mathbb Q^2, \qquad T=\begin{psmallmatrix}1&1\\0&1\end{psmallmatrix}, \quad P=\begin{psmallmatrix}0&0\\1&1\end{psmallmatrix}.
  \]
  Then for every $n\in\mathbb Z$ one has $T^n=\begin{psmallmatrix}1&n\\0&1\end{psmallmatrix}$ and $PT^nP=(n+1)P$.
  Define a functor
  \[
    F:\mathcal A\longrightarrow\operatorname{Vect}_{\mathbb Q}, \qquad F(*)=V,
    \qquad F(a)=P,
    \quad F(s)=T.
  \]
  Every $F(s^n)=T^n$ is invertible, so the universal property of localization induces a functor $H:\mathcal A[S^{-1}]\longrightarrow \operatorname{Vect}_{\mathbb Q}$.
  In particular,
  \[
    H([a])=P,
    \qquad H([s])=T,
    \qquad H([\overline s])=T^{-1}.
  \]
  Taking $n=-1$ gives $H([a\overline s a])=PT^{-1}P=0$.

  Take any $m\in\{a,s\}^*$.
  If $m=s^r$ with $r\geq0$, then $H([m])=T^r\neq0$.
  If $m$ contains $t\geq1$ copies of the letter $a$, then there exist $n_0,\ldots,n_t\in\mathbb N$ such that $m=s^{n_0}a s^{n_1}a\cdots a s^{n_t}$.
  Repeatedly applying $PT^{n_j}P=(n_j+1)P$ gives
  \[
    H([m])
    =T^{n_0}PT^{n_1}P\cdots PT^{n_t}
    =\bigg(\prod_{j=1}^{t-1}(n_j+1)\bigg)
    T^{n_0}PT^{n_t}\neq0.
  \]
  When $t=1$, the empty product here is $1$.
  Thus $H([m])$ is nonzero for every $m\in\{a,s\}^*$.
  If $[a\overline s a]\cong[m]$, applying the functor
  \[
    H^{[1]}:(\mathcal A[S^{-1}])^{[1]}
    \longrightarrow(\operatorname{Vect}_{\mathbb Q})^{[1]},
  \]
  to this isomorphism would give an isomorphism between the zero morphism $0:V\longrightarrow V$ and the nonzero morphism $H([m]):V\longrightarrow V$.
  This is a contradiction.
  Therefore $[a\overline s a]$ is not isomorphic to any object in the image of $\Theta_{[1]}$, so $\Theta_{[1]}$ is not essentially surjective.
\end{proof}

\subsection{A nonadditive localization between additive categories}\label{subsec:nonautomatic-additivity-solution}

Let $k$ be a field, and $\mathrm{mod}_k$ the category of finite-dimensional $k$-vector spaces.
Let $\mathcal A$ be the category of finite families of monomorphisms with a common codomain, described as follows.
\begin{enumerate}
  \item An object is $X=(V,(V_i)_{i\geq 1})$, where $V\in\mathrm{mod}_k$ and the $V_i$ are subspaces of $V$, and all but finitely many of the $V_i$ are zero. 
  \item For $X = (V,(V_i)_{i\geq 1})$ and $Y = (W, (W_i)_{i \geq 1})$, a morphism $f:X\longrightarrow Y$ is a linear map $f:V\longrightarrow W$ such that $f(V_i)\subseteq W_i$ for every $i$.
  Equivalently, the following square commutes for every $i$:
  \[
    \xymatrix@C=8mm@R=6mm{
      V_i\ar@{^{(}->}[r]\ar[d]_-{f_i}&V\ar[d]^-{f}\\
      W_i\ar@{^{(}->}[r]&W
    },
  \]
  where $f_i$ is the restriction of $f$ to $V_i$.
\end{enumerate}
By construction, $\mathcal A$ is an additive category.

\begin{lemma}\label{lem:nonadditive-functor}
  The following $F$ is a nonadditive functor:
  \[
  F: \mathcal A\longrightarrow\mathrm{mod}_k,\qquad
  X \longmapsto \frac{\bigwedge^2 V}{\sum_{i\geq 1}\bigwedge^2 V_i},
  \quad
  f \longmapsto \overline{\bigwedge\nolimits^2 f}.
\]
\end{lemma}

\begin{proof}
  For a morphism $f:X\longrightarrow Y$, one has $(\bigwedge^2 f)(\sum_{i\geq 1}\bigwedge^2 V_i)\subseteq \sum_{i\geq 1}\bigwedge^2 W_i$. 
  The induced map $Ff$ is given by the following diagram:
  \[
\xymatrix@C=8mm@R=6mm{
  0 \ar[r]
  & {\sum_{i\geq 1}\bigwedge^2 V_i}
    \ar[r] \ar[d]_{(\bigwedge\nolimits^2 f)|}
  & {\bigwedge\nolimits^2 V}
    \ar[r] \ar[d]_{\bigwedge\nolimits^2 f}
  & {\frac{\bigwedge\nolimits^2 V}{\sum_{i\geq 1}\bigwedge^2 V_i}}
    \ar[r] \ar@{..>}[d]^{\overline{\bigwedge\nolimits^2 f}}
  & 0 \\
  0 \ar[r]
  & {\sum_{i\geq 1}\bigwedge^2 W_i} \ar[r]
  & {\bigwedge\nolimits^2 W} \ar[r]
  & {\frac{\bigwedge\nolimits^2 W}{\sum_{i\geq 1}\bigwedge^2 W_i}}
    \ar[r]
  & 0
}.
\]
It is clear that $F$ is a functor.
Since $F((k, (0)_{i \geq 1})) = 0$ but $F((k^2, (0)_{i \geq 1})) \neq 0$, $F$ is not additive.
\end{proof}

Set $S=\bigl\{f:X\longrightarrow Y\text{ in }\mathcal A
\;\big|\;Ff:FX\longrightarrow FY\ \text{is an isomorphism}\bigr\}$.
Then the universal property of localisation $Q : \mathcal{A} \longrightarrow \mathcal{A}[S^{-1}]$ gives the induced $\widetilde F$ in the commutative diagram:
\[
\xymatrix@C=12mm@R=6mm{
  \mathcal A
    \ar[rr]^-{F}
    \ar[dr]_-{Q}
  && \mathrm{mod}_k \\
  & \mathcal A[S^{-1}]
    \ar@{..>}[ur]_-{\widetilde F}
  &
}.
\]
Note that $F$ admits a right inverse
\[
  J:\mathrm{mod}_k\longrightarrow\mathcal A,\qquad
  U\longmapsto \bigl(k\oplus U,\ (0\oplus U,\ 0,\ 0,\ldots)\bigr), \quad \varphi \longmapsto \mathrm{Id}_k \oplus \varphi.
\]
Indeed, taking $e=(1,0)\in k\oplus U$, one has the natural isomorphism $\eta : FJ\cong\operatorname{Id}_{\mathrm{mod}_k}$ via the map
\[
\eta _U: FJU \longrightarrow J, \qquad [e\wedge z]\longmapsto z.
\]

\begin{lemma}\label{lem:comparison-with-JF}
  Let $X=(V,(V_i)_{i\geq 1})$ and $Y=(W,(W_i)_{i\geq 1})$ be objects of $\mathcal A$.
  For every morphism $f:X\longrightarrow Y$, one can construct a commutative diagram, as follows
  \[
  \xymatrix@C=10mm@R=6mm{
    X \ar@{=>}[r] \ar[d]_-{f}
      & \underline X
        \ar[d]^-{f\oplus JFf}
      & JFX \ar@{=>}[l] \ar[d]^-{JFf} \\
    Y \ar@{=>}[r]
      & \underline Y
      & JFY \ar@{=>}[l].
  }
  \]
  Here $\underline X$ has the same underlying space as $X\oplus JFX$, i.e., $V \oplus (k \oplus FX)$.
  The associated subspaces of $\underline X$ are of the following three types:
  \begin{enumerate}
    \item the subspaces of $X$ and $JFX$, via the canonical inclusions $X\longrightarrow X\oplus JFX$ and $JFX\longrightarrow X\oplus JFX$,
    \item let $(v_j)_{j=1}^m$ and $(z_l)_{l=1}^r$ be a basis of $V$ and $FX$, respectively, then the following subspaces are added:
    \[
\begin{aligned}
&\operatorname{span}_k\{(v_j,0,0),(0,1,0)\}
&& (1\le j\le m),\\
&\operatorname{span}_k\{(v_j,0,0),(0,0,z_l )\}
&& (1\le j\le m,\ 1\le l\le r).
\end{aligned}
\]
\item let $(v_j)_{j=1}^m$ be a basis of $V$, then the following subspaces are added:
\[
\operatorname{span}_k
\bigl\{(v_j,1,0),\ (v_l,0,-[v_j\wedge v_l])\bigr\} \qquad (1 \leq j < l \leq m).
\]
  \end{enumerate}
  Similarly, one defines
  $\underline Y$.
  The horizontal double arrows are induced by structural inclusions.
\end{lemma}

\begin{proof}
  We first shows that the horizontal double arrows are in $S$.
  For $X \Longrightarrow \underline X$, the morphism restricts to the underlying space is $V \longrightarrow V\oplus(k\oplus FX), \quad v\longmapsto(v,0,0)$.
    Applying $F$ to this inclusion $X \Longrightarrow \underline X$, one obtains 
    \[
    {\frac{\bigwedge^2 V}{\sum_{i\geq1}\bigwedge^2 V_i}}
\longrightarrow
{\frac{\bigwedge^2\!\bigl(V\oplus(k\oplus FX)\bigr)}
{\sum_{i\geq1}\bigwedge^2\underline {V_i}}},
\qquad
[x\wedge y]\longmapsto
[(x,0,0)\wedge(y,0,0)].
\]
We claim that this map is an isomorphism.
\begin{itemize}
  \item (Surjectivity). By construction, $[(x,0,0)\wedge(y,0,0)]$ has a preimage $[x\wedge y]$.
  $[(v,0,0) \wedge (0,1,0)]$, $[(v,0,0) \wedge (0,0,z)]$, $[(0,1,0) \wedge (0,1,0)]$, and $[(0,0,z)\wedge(0,0,z')]$ vanish.
  By construction, $[(0,1,0)\wedge(0,0,z)]$ is the image of $z\in FX$, where $z$ is represented by the linear combinations of $[x\wedge y]$'s.
  \item (Injectivity). We construct a left inverse of the morphism.
  The alternating bilinear map
\[
  \bigl((x,a,z),(y,b,z')\bigr)
  \longmapsto [x\wedge y]+az'-bz
\]
induces a linear map
$\bigwedge^2(V\oplus k\oplus FX)\longrightarrow FX$, which kills all three types of subspaces in the construction of $\underline X$.
Thus the map descends to $F\underline X\longrightarrow FX$.
The composite $FX\longrightarrow F\underline X\longrightarrow FX$ sends $[x\wedge y]$ to itself, and hence $FX\longrightarrow F\underline X$ is injective.
\end{itemize}
Note that $F(JFX) \longrightarrow F \underline X \longrightarrow FX$ coincides with $\eta_{FX}$, which is an isomorphism.
Hence $JFX \Longrightarrow \underline X$ is also an isomorphism under $F$.
The same argument applies to $Y \Longrightarrow \underline Y$ and $JFY \Longrightarrow \underline Y$.
The commutativity of the diagram is straightforward to check.
\end{proof}

\begin{lemma}\label{lem:nonautomatic-additivity-equivalence}
  The category $\mathcal A[S^{-1}]$ is additive, and $\widetilde F$ is an additive equivalence.
\end{lemma}

\begin{proof}
  Since $\widetilde FQJ=FJ\cong\operatorname{Id}_{\mathrm{mod}_k}$, the functor $\widetilde F$ is essentially surjective.
  Given a linear map $FX\longrightarrow FY$, apply $QJ$ and use the isomorphisms at $QX$ and $QY$ from Lemma~\ref{lem:comparison-with-JF}. The resulting map $QX\longrightarrow QY$ has the prescribed image under $\widetilde F$, so $\widetilde F$ is full.

  It remains to show that $\widetilde F$ is faithful.
  Consider two zigzags in $\mathcal A$ with $s^{(i)},t^{(i)}\in S$
  \[
\xymatrix@C=15mm@R=0mm{
  & \cdot
  & \cdot \ar@{=>}[l]_-{s^{(1)}}
          \ar[r]^-{f^{(1)}}
  & \cdots \ar[r]^-{f^{(n-1)}}
  & \cdot
  & \cdot \ar@{=>}[l]_-{s^{(n)}}
          \ar[dr]^-{f^{(n)}}
  & \\
  X \ar[ur]^-{f^{(0)}}
             \ar[dr]_-{g^{(0)}}
  &&&&&& Y \\
  & \cdot
  & \cdot \ar@{=>}[l]^-{t^{(1)}}
          \ar[r]_-{g^{(1)}}
  & \cdots \ar[r]_-{g^{(n-1)}}
  & \cdot
  & \cdot \ar@{=>}[l]^-{t^{(n)}}
          \ar[ur]_-{g^{(n)}}
  &
}
  \]
  and suppose that the two zigzags give the same map $FX\longrightarrow FY$ after applying $F$.
  By Lemma~\ref{lem:comparison-with-JF}, one can construct a comparison diagram for each arrow in the zigzags.
  \[
  \xymatrix@C=13mm@R=3mm{
    QJFX \ar@{=>}[d] \ar[r]^-{QJ(Ff^{(0)})}
    & \cdot \ar@{=>}[d]
    & \cdot \ar@{=>}[d] \ar@{=>}[l]_-{QJ(Fs^{(1)})}
            \ar[r]^-{QJ(Ff^{(1)})}
    & \cdots \ar[r]^-{QJ(Ff^{(n-1)})}
    & \cdot \ar@{=>}[d]
    & \cdot \ar@{=>}[d] \ar@{=>}[l]_-{QJ(Fs^{(n)})}
            \ar[r]^-{QJ(Ff^{(n)})}
    & QJFY \ar@{=>}[d] \\
    Q\underline X \ar[r]
    & Q\underline{\cdot}
    & Q\underline{\cdot} \ar@{=>}[l] \ar[r]
    & \cdots \ar[r]
    & Q\underline{\cdot}
    & Q\underline{\cdot} \ar@{=>}[l] \ar[r]
    & Q\underline Y \\
    & \cdot \ar@{=>}[u]
    & \cdot \ar@{=>}[u] \ar@{=>}[l]_-{Q(s^{(1)})}
            \ar[r]^-{Q(f^{(1)})}
    & \cdots \ar[r]^-{Q(f^{(n-1)})}
    & \cdot \ar@{=>}[u]
    & \cdot \ar@{=>}[u] \ar@{=>}[l]_-{Q(s^{(n)})}
            \ar[dr]^-{Q(f^{(n)})}
    & \\
    QX \ar@{=>}[uu] \ar@{=>}[dd]
       \ar[ur]^-{Q(f^{(0)})} \ar[dr]_-{Q(g^{(0)})}
    &&&&&& QY \ar@{=>}[uu] \ar@{=>}[dd] \\
    & \cdot \ar@{=>}[d]
    & \cdot \ar@{=>}[d] \ar@{=>}[l]^-{Q(t^{(1)})}
            \ar[r]_-{Q(g^{(1)})}
    & \cdots \ar[r]_-{Q(g^{(n-1)})}
    & \cdot \ar@{=>}[d]
    & \cdot \ar@{=>}[d] \ar@{=>}[l]^-{Q(t^{(n)})}
            \ar[ur]_-{Q(g^{(n)})}
    & \\
    Q\underline X \ar[r]
    & Q\underline{\cdot}
    & Q\underline{\cdot} \ar@{=>}[l] \ar[r]
    & \cdots \ar[r]
    & Q\underline{\cdot}
    & Q\underline{\cdot} \ar@{=>}[l] \ar[r]
    & Q\underline Y \\
    QJFX \ar@{=>}[u] \ar[r]_-{QJ(Fg^{(0)})}
    & \cdot \ar@{=>}[u]
    & \cdot \ar@{=>}[u] \ar@{=>}[l]^-{QJ(Ft^{(1)})}
            \ar[r]_-{QJ(Fg^{(1)})}
    & \cdots \ar[r]_-{QJ(Fg^{(n-1)})}
    & \cdot \ar@{=>}[u]
    & \cdot \ar@{=>}[u] \ar@{=>}[l]^-{QJ(Ft^{(n)})}
            \ar[r]_-{QJ(Fg^{(n)})}
    & QJFY \ar@{=>}[u]
  }.
  \]
  Note that all $2 \times 2$ squares in this diagrams commutes.
  By assumption, the outer paths from $QJFX$ to $QJFY$ are equal.
  Thus the two paths from $QX$ to $QY$ are equal, and hence $\widetilde F$ is faithful.
\end{proof}

The category $\mathcal A$ is additive by construction, and Lemma~\ref{lem:nonautomatic-additivity-equivalence} shows that that $\widetilde F$ is an equivalence, and thus $\mathcal A[S^{-1}]$ is an additive category.
The localization functor $Q$ is not additive.
Otherwise, if $Q$ is additive, then $F=\widetilde F\circ Q$ is also additive, contrary to Lemma~\ref{lem:nonadditive-functor}.
This proves Conjecture~\ref{conj:nonautomatic-additivity}.

\section*{Use of AI}
The manuscript was completed several months before the AI-assisted checks described here.
GPT-5.6 sol was used to locate references and check mathematical statements.
GPT-6 astra was also used in an attempt to resolve Conjecture~7.1, but this attempt yielded no progress.

\end{document}